\documentclass[12pt,a4paper]{article}
\usepackage{geometry}
\usepackage{hyperref}

\usepackage{comment}

\usepackage{times,authblk}
\usepackage[UKenglish]{babel}
\usepackage{combelow}
\renewcommand{\phi}{\varphi}
\renewcommand{\rho}{\varrho}
\renewcommand{\epsilon}{\varepsilon}
\usepackage[cal=esstix]{mathalfa}  
\usepackage{amsmath, amsfonts, amssymb, amsthm, bm, stmaryrd, enumerate, mathtools}
\usepackage{graphicx}
\usepackage{tikz}
\usepackage{tikz-cd}
\tikzset{commutative diagrams/.cd}

\newtheorem{Def}{Definition}[section]
\newenvironment{definition}{\begin{Def} \rm}{\end{Def}}
\newtheorem{lemma}[Def]{Lemma}
\newtheorem{proposition}[Def]{Proposition}
\newtheorem{corollary}[Def]{Corollary}
\newtheorem{theorem}[Def]{Theorem}
\newtheorem{example}[Def]{Example}
\newtheorem{remark}[Def]{Remark}

\newcommand{\komma}{,\hspace{0.3em}}
\newcommand{\id}{\text{\rm id}}
\renewcommand{\leq}{\leqslant}
\renewcommand{\geq}{\geqslant}

\newenvironment{sm}{\scriptsize \begin{pmatrix}}{\end{pmatrix}}

\newcommand{\Naturals}{{\mathbb N}}

\newcommand{\Reals}{{\mathbb R}}
\newcommand{\Complexes}{{\mathbb C}}
\newcommand{\Quaternions}{{\mathbb H}}

\newcommand{\notperp}{\mathbin{\not\perp}}

\newcommand{\Zero}{{\bf 0}}
\renewcommand{\c}{^\perp}
\newcommand{\cc}{^{\perp\perp}}

\newcommand{\ce}[1]{^{\perp_{#1}}}
\newcommand{\cce}[1]{^{\perp_{#1}\perp_{#1}}}
\newcommand{\herm}[2]{\left( #1 , #2 \right)}
\newcommand{\lin}[1]{\langle #1\rangle}
\renewcommand{\P}{\mathbf P}
\DeclareMathOperator{\Sym}{Sym}
\newcommand{\scmu}{\mathbin{\raisebox{2pt}{$\scriptscriptstyle\diamond$}}}
\newcommand{\prl}{\parallel}

\newcommand{\class}[1]{[#1]}
\newcommand{\zerokernel}[1]{{#1}^{\circ}}

\newcommand{\withoutzero}{^{\raisebox{0.2ex}{\scalebox{0.4}{$\bullet$}}}}
\newcommand{\adj}{^\ast}

\newcommand{\inv}{^\ast}

\newcommand{\OMS}[1]{\mathcal{OMS}_{#1}}
\newcommand{\OMSfin}[1]{\mathcal{OMS}^\text{\rm fin}_{#1}}
\newcommand{\Hil}[1]{\mathcal{Hil}_{#1}}
\newcommand{\C}{{\mathcal C}}
\newcommand{\Cfin}{{\mathcal C}^\text{\rm fin}}
\renewcommand{\L}{{\mathsf L}}
\newcommand{\Lfin}{{\mathsf L}^\text{\rm fin}}
\newcommand{\E}{{\mathsf E}}

\newcommand{\op}{^\text{op}}
\DeclareMathOperator{\kernel}{ker}
\DeclareMathOperator{\image}{im}
\DeclareMathOperator{\homset}{Hom}

\DeclareMathOperator{\spectrum}{sp}

\begin{document}

\title{Dagger categories of orthosets \\ and the complex Hilbert spaces}

\author[1]{Jan Paseka}

\author[2]{Thomas Vetterlein}

\affil[1]{\footnotesize Department of Mathematics and Statistics,
Masaryk University \authorcr
Kotl\'a\v rsk\'a 2, 611\,37 Brno, Czech Republic \authorcr
{\tt paseka@math.muni.cz}}

\affil[2]{\footnotesize Institute for Mathematical Methods in Medicine and Data Based Modeling, \authorcr
Johannes Kepler University Linz \authorcr
Altenberger Stra\ss{}e 69, 4040 Linz, Austria \authorcr
{\tt Thomas.Vetterlein@jku.at}}

\date{\today}

\maketitle

\begin{abstract}\parindent0pt\parskip1ex

\noindent An orthoset is a non-empty set $X$ together with a symmetric binary relation $\perp$ and a constant $0$ such that $x \notperp x$ for any $x \neq 0$, and $0 \perp x$ for any $x$. Maps $f \colon X \to Y$ and $g \colon Y \to X$ between orthosets are said to form an adjoint pair if, for any $x \in X$ and $y \in Y$, $f(x) \perp g$ if and only if $x \perp g(x)$. Hilbert spaces, equipped with the usual orthogonality relation and the zero vector, provide the motivating examples of orthosets. The usual adjoints of bounded linear maps between Hilbert spaces are adjoints also in our sense.

We investigate dagger categories of orthosets and maps between them, requiring that any morphism and its dagger form an adjoint pair. We indicate conditions under which such a category is unitarily dagger equivalent to the dagger category of complex Hilbert spaces and bounded linear maps.



\mbox{}\vspace{-2ex}

\end{abstract}

\section{Introduction}
\label{sec:Introduction}

The complex Hilbert space provides a mathematical framework for modelling quantum phenomena. Its vectors, its linear structure and the inner product do not themselves reflect form and substance of physical observations. The structure of the physical observables is an aspect that comes in addition; in particular, physical symmetries correspond to unitary group actions. In order to justify the choice of a Hilbert space as a mathematical model, we thus cannot simply refer to intuition. Rather, we should consider the Hilbert space as a tool that is, in a sense, universally applicable for representing the various kinds of structures that arise in physics. The question seems natural whether we can give a plausible reason for its usefulness and versatility.

It would be helpful to have a means to derive the Hilbert space model from structures that are of an elementary nature and easy to accept as the starting point for physical modelling. The concern of structural reduction has been discussed in the framework of the so-called logico-algebraic approach to the foundation of quantum mechanics. The idea is to associate with a Hilbert space a sort of ``logic''. Indeed, statements about a quantum physical system that can be verified by experiment correspond to the closed subspaces and the collection of all closed subspaces forms an ortholattice. Moreover, the Hilbert space can be reconstructed from this ortholattice. The efforts to describe the Hilbert space lattice-theoretically were to some extent successful, but the conditions are technically involved \cite{Wlb}. Trying to go to the limits, one may instead focus solely on the binary relation of orthogonality among the $1$-dimensional subspaces of a Hilbert space, which model the pure states of a quantum physical system. To this end, David Foulis and his collaborators once suggested a graph-theoretical approach. They coined the notion of an {\it orthogonality space}, defined to be a set equipped with a symmetric, irreflexive binary relation $\perp$, called the {\it orthogonality relation} \cite{Dac,Wlc}. An orthogonality space is thus nothing but an undirected graph. To describe the complex Hilbert space as such a structure, however, is again not a straightforward matter \cite{Vet1,Vet2}.

The mentioned endeavours to characterise the Hilbert space have focused on a single object in isolation -- a certain ortholattice, a certain orthogonality space. In the present paper, we go instead the ``categorical'' way, not considering the single objects as most important but the mutual interrelation between them. We are inspired by the categorical approach to the discussion on quantum foundations, which was initiated by S.~Abramsky and B.~Coecke \cite{AbCo} and has become a very active field of research since then \cite{HeVi}. In our context particularly remarkably, C.~Heunen and A.~Kornell succeeded to describe the category of Hilbert space in purely categorical terms \cite{HeKo}.

The present contribution is to be seen in the context of both the logico-algebraic and the categorical approach to the foundational issues in quantum physics. We focus on concrete objects, namely, on orthosets. An orthoset is essentially the same as an orthogonality space in the sense of Foulis. However, we slightly modify the definition, requiring the presence of an additional element $0$ that is assumed to be orthogonal to all elements. We moreover propagate a particular sort of maps to take on the role of morphisms. A map $f \colon X \to Y$ between orthosets is called {\it adjointable} if there is a further map $g \colon Y \to X$ such that, for any $x \in X$ and $y \in Y$, $f(x) \perp y$ if and only if $x \perp g(y)$. The map $g$ is in this case an {\it adjoint} of $f$. A collection of orthosets together with the adjointable maps between them form a category. Orthosets and adjointable maps are discussed in our previous papers \cite{PaVe4,PaVe5}, which have laid the basis of the present work.

The category $\Hil{\Complexes}$ of complex Hilbert space and bounded linear maps is an example of a category of this sort. Indeed, each complex Hilbert space is an orthoset: the orthogonality relation is the familiar one and the $0$ is the zero vector. Moreover, any bounded linear map $\phi \colon H_1 \to H_2$ between Hilbert spaces is, when seen as a map between orthosets, adjointable. In fact, $\phi$ possesses the adjoint $\phi\adj \colon H_2 \to H_1$, the adjoint of $\phi$ in the common sense. The assignment $\phi \mapsto \phi\adj$ gives rise to an involutive contravariant endofunctor, or a {\it dagger}, on $\Hil{\Complexes}$. In other words, $\Hil{\Complexes}$ is actually a dagger category \cite{Sel}. This paper includes a characterisation of $\Hil{\Complexes}$ as a dagger category of orthosets.

Our approach is as follows. We consider a category $\C$ whose objects are (certain) orthosets, whose morphisms are (certain adjointable) maps between them, and whose dagger assigns to each map an adjoint. We start with an investigation of the consequences of three conditions on $\C$ that we consider as basic. Roughly speaking, the first condition describes how to compose two objects to a new one, the second one how to decompose an object into two, and according to the third one there is a unique ``atomic'' object. These hypotheses are sufficient to show that the objects of $\C$ are orthosets arising from orthomodular spaces over a certain $\star$-sfield (involution skew field) $F$. We say that an orthoset $X$ has rank $n \in \Naturals$ if the maximal number of mutually orthogonal non-zero elements of $X$ is $n$, and if there is no such $n$ we say that $X$ has infinite rank. If $\C$ contains an orthoset of infinite rank, we can, thanks to Sol\` er's Theorem, be even more specific: the $\C$-objects then arise from Hilbert spaces over one of the classical $\star$-sfields $\Reals$, $\Complexes$, or $\Quaternions$.

In a next step, we extend our hypotheses. We sharpen the assumption on the ``atomic'' object and we add a condition saying that, for each pair of objects, one is contained as a subspace in the other one. For a $\star$-sfield $F$, let $\OMS{F}$ be the dagger category of uniform orthomodular spaces over $F$ and adjointable linear maps. Here, an orthomodular space is called {\it uniform} if each $1$-dimensional subspace contains a unit vector. Moreover, let $\OMSfin{F}$ be the full subcategory of $\OMS{F}$ consisting of the finite-dimensional spaces. In the case when the orthosets in $\C$ have finite rank and under a certain condition on the $\star$-sfield (fulfilled by $\Reals$, $\Complexes$, $\Quaternions$), we show that $\C$ is unitarily dagger equivalent to $\OMSfin{F}$. We finally add a fifth condition according to which any dagger automorphism $f$ has a strict square root, that is, given $f$ there is a further dagger automorphism $g$ such that $g^2 = f$, and $f$ and $g$ commute with the same projections. We prove that, provided that $\C$ contains an orthoset of infinite rank, $\C$ is unitarily dagger equivalent with $\OMS{\Complexes}$, which coincides with $\Hil{\Complexes}$. That is, $\C$ then consists of complex Hilbert spaces and bounded linear maps.

The results of this paper require rather extensive preparations. The preliminaries are developed in Sections~\ref{sec:orthosets}--\ref{sec:dagger-categories}. We start by compiling the relevant facts about orthosets in Section~\ref{sec:orthosets} and about adjointable maps in Section~\ref{sec:adjointable-maps}. Orthosets are intended to represent orthomodular spaces, which are discussed in Sections~\ref{sec:orthomodular-spaces} and~\ref{sec:adjointable-linear-maps}. Finally, Section~\ref{sec:dagger-categories} provides the necessary background regarding dagger categories.

Sections~\ref{sec:three-basic-hypotheses}--\ref{sec:extended-set-of-hypotheses} constitute the main part of the paper. In Section~\ref{sec:three-basic-hypotheses}, we show that three basic conditions are enough to show that a dagger category of orthosets and adjointable maps can be regarded as consisting of orthomodular spaces. In Section~\ref{sec:the-functor-L}, we construct a functor $\L \colon \C \to \OMS{F}$, where $F$ is a $\star$-sfield, and we investigate its properties. In Section~\ref{sec:extended-set-of-hypotheses}, we sharpen our set of hypotheses and establish in particular that the functor $\L$ is faithful. Under the additional assumptions that dagger automorphisms possess strict square roots and that $\C$ contains an orthoset of infinite rank, we prove the unitary dagger equivalence of $\C$ with the dagger category of complex Hilbert spaces.

\section{Orthosets}
\label{sec:orthosets}

Following Foulis and his collaborators, an orthogonality space is a structure based on a binary relation $\perp$ that is supposed to be symmetric and irreflexive. Although our motivations correspond to those of Foulis, we deviate slightly from the original definition. Namely, we assume the presence of an additional element $0$ that is orthogonal to all elements. Apart from that, we use the more compact notion of an ``orthoset'' instead of an ``orthogonality space''.

\begin{definition}
An {\it orthoset} is a non-empty set $X$ equipped with a binary relation $\perp$ called the {\it orthogonality relation} and with a constant $0$ called {\it falsity}. The following properties are assumed to hold:
\begin{itemize}

\item[\rm (O1)] $x \perp y$ implies $y \perp x$ for any $x, y \in X$,

\item[\rm (O2)] $x \perp x$ only if $x = 0$, 

\item[\rm (O3)] $0 \perp x$ for any $x \in X$.

\end{itemize}
An element of $X$ distinct from falsity is called {\it proper}; we put $X\withoutzero = X \setminus \{0\}$. Elements $x, y \in X$ such that $x \perp y$ are called {\it orthogonal}.
\end{definition}

The motivating examples of orthosets arise from Hilbert spaces.

\begin{example} \label{ex:basic-1}
Let $H$ be a complex Hilbert space. Then $H$, equipped with the usual orthogonality relation and with the zero vector as falsity,  is an orthoset.

We denote the subspace spanned by a vector $u \in H$ by $\lin u$. Let us define
\[ \P(H) \;=\; \{ \lin u \colon u \in H \}. \]
That is, $\P(H)$ is the set of all at most $1$-dimensional subspaces of $H$. Equipped with the usual orthogonality relation and with the zero linear subspace as falsity, $\P(H)$ is likewise an orthoset.
\end{example}

We will compile some definitions and facts around orthosets needed for this work. To begin with, let us review what could be considered as ``separation axioms'' for orthosets.

\begin{definition} \label{def:separation-axioms}
Let $X$ be an orthoset.
\begin{itemize}

\item[\rm (i)] We call $X$ {\it irredundant} if, for any distinct proper elements $x, y \in X$, there is a $z \in X$ such that either $z \perp x$ but $z \notperp y$, or $z \perp y$ but $z \notperp x$.

\item[\rm (ii)] We call $X$ {\it atomistic} if, for any proper elements $x, y \in X$, the following holds: If there is a $z \in X$ such that $z \perp x$ but $z \notperp y$, then there is also an $z' \in X$ such that $z' \perp y$ but $z' \notperp x$.

\end{itemize}
\end{definition}

For an element $x$ of an orthoset $X$, let $\{x\}\c$ be the set of all elements orthogonal to $x$. We observe that $X$ is irredundant if, for any $x, y \in X\withoutzero$, $\{x\}\c = \{y\}\c$ implies $x = y$, and $X$ is atomistic if, for any $x, y \in X\withoutzero$, $\{x\}\c \subseteq \{y\}\c$ implies $\{x\}\c = \{y\}\c$.

We may associate with any orthoset an irredundant one as follows. Let us call two elements $x$ and $y$ of $X$ {\it equivalent} if $\{x\}\c = \{y\}\c$; we write $x \prl y$ in this case. We denote the equivalence class of some $x \in X$ by $\class x$ and we call
\[ P(X) \;=\; \{ \class x \colon x \in X \} \]
the {\it irredundant quotient} of $X$. In the obvious sense, $P(X)$ is an irredundant orthoset \cite[Proposition 2.8]{PaVe4}.

\begin{example} \label{ex:basic-2}
Let $H$ be a Hilbert space. In accordance with Example~\ref{ex:basic-1}, we view $H$ as an orthoset. We observe that $H$ is atomistic.

For $u \in H$, we have
\begin{equation} \label{fml:basic-2}
\class u \;=\; \{ \alpha \, u \colon \alpha \neq 0 \}
\;=\; \begin{cases} \lin u \setminus \{0\} & \text{if $u \neq 0$,} \\
\{0\} & \text{if $u = 0$} \end{cases}
\end{equation}
The irredundant quotient $P(H) = \{ \class u \colon u \in H \}$ can thus can be identified with $\P(H)$, the collection of subspaces of $H$ spanned by single vectors. Indeed, we obtain $P(H)$ from $\P(H)$ by removing the zero vector from every non-zero subspace.
\end{example}

Orthosets can be described to a good extent by lattice-theoretic tools. The simple idea is as follows. For any subset $A$ of an orthoset $X$, the {\it orthocomplement} of $A$ is
\[ A\c \;=\; \{ x \in X \colon x \perp y \text{ for all } y \in A \}. \]
The double orthocomplementation $\cc$ is a closure operator on $X$ and we call the sets $A \subseteq X$ such that $A = A\cc$ {\it orthoclosed}. We note that $0$ is contained in any orthoclosed set. We denote the collection of all orthoclosed subsets of $X$ by ${\mathsf C}(X)$. Partially ordered by set-theoretic inclusion and equipped with the orthocomplementation~$\c$, ${\mathsf C}(X)$ is a complete ortholattice, $\{0\}$ and $X$ being the bottom and top elements.

By a {\it subspace} of an orthoset $X$, we mean an orthoclosed set $A \subseteq X$ equipped with the restriction of the orthogonality relation and the constant $0$. Evidently, $A$ becomes in this way an orthoset as well. The orthocomplementation on $A$ will in this case be denoted by $\ce{A}$, that is, for $B \subseteq A$ we put $B\ce{A} = B\c \cap A$.

Let $A_1, \ldots, A_k$ be subspaces of $X$. We call $(A_1, \ldots, A_k)$ a {\it decomposition} of $X$ if $A_i = \big(\bigvee_{j \neq i} A_j\big)\c$ for each $i = 1, \ldots, k$ \cite{PaVe2}. In case when $k=2$, $A_1$ and $A_2$ are mutual orthocomplements and we refer to them as {\it complementary subspaces}.

Finally, an orthoset $X$ is called {\it reducible} if $X$ possesses complementary subspaces $A$ and $B$ such that both contain a proper element and $A \cup B = X$. Otherwise, we say that $X$ is {\it irreducible}.

\begin{lemma} \label{lem:Dacey-space}
Let $X$ be an orthoset. Then the following are equivalent:
\begin{itemize}

\item[\rm (a)] For any collection $A_1, \ldots, A_k$ of mutually orthogonal subspaces of $X$ whose join is $X$, $(A_1, \ldots, A_k)$ is a decomposition.

\item[\rm (b)] ${\mathsf C}(X)$ is an orthomodular lattice.

\item[\rm (c)] For any subspace $A$ of $X$ and any $B \subseteq A$, we have $B\cce{A} = B\cc$.

\item[\rm (d)] For any subspace $A$ of $X$, ${\mathsf C}(A) = \{ B \in {\mathsf C}(X) \colon B \subseteq A \}$.

\item[\rm (e)] For any $A \in {\mathsf C}(X)$ and any maximal $\perp$-set $D$ contained in $A$, we have that $A = D\cc$.

\end{itemize}
\end{lemma}

\begin{proof}
(a) $\Rightarrow$ (b): Let $A, B \in {\mathsf C}(X)$ such that $B \subseteq A$. Then $A\c$, $B$, and $A \cap B\c$ are mutually orthogonal subspaces whose join is $X$. Hence (a) implies that $(A\c, B, A \cap B\c)$ is a decomposition of $X$, and this means $A = B \vee (A \cap B\c)$.

(b) $\Rightarrow$ (a): Mutually orthogonal elements of an orthomodular lattice generate a Boolean subalgebra.

For the equivalence of (b), (c), and (d), see \cite[Lemma 2.15]{PaVe4}. Finally, for the equivalence of (b) with {\it Dacey's criterion} (e), see, e.g., \cite{Wlc}.
\end{proof}

An orthoset $X$ fulfilling one of the equivalent properties of Lemma~\ref{lem:Dacey-space} is called a {\it Dacey space}. Note that in this case $P(X)$ is Dacey as well, because ${\mathsf C}(X)$ is isomorphic to ${\mathsf C}(P(X))$. Moreover, as seen in the next two lemmas, properties of Dacey spaces are often ``inherited'' to their subspaces.

We recall that a lattice with $0$ is said to have the {\it covering property} if, for any element $a$ and atom $p \nleq a$, $\,a \vee p$ covers $a$. By an {\it AC} lattice, we mean an atomistic lattice with the covering property.

\begin{lemma} \label{lem:subspaces-of-Dacey-spaces}
Let $A$ be a subspace of a Dacey space $X$.
\begin{itemize}

\item[\rm (i)] $A$ is Dacey as well.

\item[\rm (iii)] If $X$ is atomistic, so is $A$.

\item[\rm (iv)] If ${\mathsf C}(X)$ has the covering property, so has ${\mathsf C}(A)$.

\end{itemize}
\end{lemma}

\begin{proof}
See \cite[Lemma 2.16]{PaVe4}.
\end{proof}

By a {\it $\perp$-set}, we mean a subset of an orthoset $X$ consisting of mutually orthogonal proper elements. The supremum of the cardinalities of $\perp$-sets in $X$ is called the {\it rank} of $X$. We tacitly assume in this paper that all orthosets have at most countable rank. We note that the rank of the irredundant quotient $P(X)$ coincides with the rank of $X$.

An orthoset of rank $0$ consists of the falsity alone. We will refer to the orthoset $\Zero = \{0\}$ as the {\it zero orthoset}. Moreover, an orthoset $X$ of rank $1$ does not contain any pair of orthogonal proper elements, hence $X$ is just a set with at least two elements one of which is designated as the falsity. We call $X$ in this case a {\it singleton}.

An orthoset $X$ possessing a proper element $z$ such that no proper element of $X$ is orthogonal to $z$ will be called {\it unital}. Obviously, we have $X = \{z\}\cc$ in this case. Note that any singleton is a unital orthoset.

\begin{lemma} \label{lem:unital}
Let $X$ be an orthoset.
\begin{itemize}

\item[\rm (i)] For any $x \in X\withoutzero$, the subspace $\{x\}\cc$ is unital.

If $X$ is Dacey, each unital subspace is of this form.

\item[\rm (ii)] Let $X$ be atomistic. Then $\{x\}\cc$ is a singleton for any $x \in X\withoutzero$ and in fact we have $\{x\}\cc = \class x \cup \{0\}$. Moreover, $\{ \{x\}\cc \colon x \in X\withoutzero \}$ is the set of atoms of ${\mathsf C}(X)$.

\item[\rm (iii)] $X$ is atomistic and unital if and only if $X$ is a singleton.

\end{itemize}
\end{lemma}

\begin{proof}
Ad (i): Put $A = \{x\}\cc$. Then $\{x\}\ce{A} = \{0\}$ and hence $A$ is unital.

If $X$ is Dacey and $A$ a unital subspace of $X$, then there is a $z \in A$ such that $A = \{z\}\cce{A}$. By criterion (c) of Lemma~\ref{lem:Dacey-space}, we then have $A = \{z\}\cc$.

Ad (ii): Let $x \in X\withoutzero$. Then, for any $y \in X\withoutzero$, $y \in \{x\}\cc$ iff $\{y\}\cc \subseteq \{x\}\cc$ iff $\{x\}\c \subseteq \{y\}\c$ iff $\{x\}\c = \{y\}\c$ iff $x \prl y$. Hence $\{x\}\cc = \class x \cup \{0\}$.

This also shows that $\{x\}\cc$ is an atom of ${\mathsf C}(X)$. Clearly, any atom of ${\mathsf C}(X)$ is of this form.

Ad (iii): The ``only if'' part is clear from part (ii), and the ``if'' part is obvious.
\end{proof}

In what follows, the sum of cardinal numbers is understood in the usual set-theoretical way.

\begin{lemma} \label{lem:rank-of-subspaces-of-Dacey-spaces}
Let $X$ be an atomistic Dacey space and assume that ${\mathsf C}(X)$ has the covering property.
\begin{itemize}

\item[\rm (i)] Let $D$ be a maximal $\perp$-set in a subspace $A$ of $X$. Then the rank of $A$ is the cardinality of $D$.

\item[\rm (ii)] Let $(A,B)$ be a decomposition of $X$. Then the rank of $X$ is the sum of the rank of $A$ and the rank of $B$.

\item[\rm (iii)] Let $X$ be of rank $n \in \Naturals \cup \{\aleph_0\}$ and let $n_1, n_2 \in \Naturals \cup \{\aleph_0\}$ be such that $n = n_1 + n_2$. Then $X$ possesses a decomposition $(A,B)$ such that $A$ has rank $n_1$ and $B$ has rank $n_2$.

\end{itemize}
\end{lemma}

\begin{proof}
Ad (i): By Lemma~\ref{lem:subspaces-of-Dacey-spaces}, it suffices to show the assertion for $A = X$. So let $D$ be a maximal $\perp$-set in $X$.

As we assume all orthosets to have rank $\leq \aleph_0$, $D$ is finite or countably infinite. In the latter case, $X$ clearly has rank $\aleph_0$. Assume that $D$ is finite. By Lemma~\ref{lem:unital}(ii), $\{ \{x\}\cc \colon x \in D\}$, is a finite collection of atoms of ${\mathsf C}(X)$. Moreover, $\{x\}\cc \cap \bigvee_{y \in D \setminus \{x\}} \{y\}\cc = \{0\}$ for any $x \in D$ and $\bigvee_{x \in D} \{x\}\cc = D\cc = X$. As ${\mathsf C}(X)$ has the covering property, any collection of atoms with the same two properties is finite and has in fact the same cardinality as $D$ \cite[Lemma (8.3)]{MaMa}. The assertion follows.

Ad (ii): Let $D$ be a maximal $\perp$-set in $A$ and $E$ a maximal $\perp$-set in $B$. Then $D \cup E$ is a maximal $\perp$-set in $X$ because, by criterion (e) of Lemma~\ref{lem:Dacey-space}, $X = A \vee B = D\cc \vee E\cc = (D \cup E)\cc$, that is, $(D \cup E)\c = \{0\}$. Hence the assertion follows from part (i).

Ad (iii): By part (i), any maximal $\perp$-set has $n$ elements. Hence there are $D, E \subseteq X$ such that $D \cup E$ is a maximal $\perp$-set, $D$ has $n_1$ elements, and $E$ has $n_2$ elements. Putting $A = D\cc$ and $B = E\cc$, we have that $(A,B)$ is a decomposition of $X$. Moreover, $D$ is a maximal $\perp$-set in $A$ and hence $A$ has rank $n_1$ by part (i). Similarly, $B$ has rank $n_2$.
\end{proof}

\section{Adjointable maps}
\label{sec:adjointable-maps}

In this section, we compile some basic facts about the maps between orthosets that we will choose as morphisms of categories. A more detailed discussion, including most of the proofs of this section, can be found in \cite{PaVe4}.

In what follows, $X$ and $Y$ are orthosets.

\begin{definition} \label{def:adjoint}
We say that $g \colon Y \to X$ is an {\it adjoint} of $f \colon X \to Y$ if, for any $x \in X$ and $y \in Y$
\[ f(x) \perp y \quad\text{if and only if}\quad x \perp g(y). \]
Moreover, a map $f \colon X \to X$ is called {\it self-adjoint} if $f$ is an adjoint of itself.
\end{definition}

We will exclusively deal with maps between orthosets that possess an adjoint and we call these maps {\it adjointable}. Note that adjointness is a symmetric property and we will occasionally speak of two mutually adjoint maps as an {\it adjoint pair}.

Clearly, the identity map $\id \colon X \to X$ is self-adjoint, and if $f\adj$ is an adjoint of $f \colon X \to Y$ and $g\adj$ is an adjoint of $g \colon Y \to Z$, then $f\adj \circ g\adj$ is an adjoint of $g \circ f$.

We call two maps $f, f' \colon X \to Y$ between orthosets {\it equivalent} if $f(x) \prl f'(x)$ for all $x \in X$. We write $f \prl f'$ in this case. The next lemma shows that adjoints are uniquely determined up to equivalence.

\begin{lemma} \label{lem:irredundancy-adjoints-1}
Let $f \colon X \to Y$ and $g \colon Y \to X$ be an adjoint pair of maps. Then $g' \colon Y \to X$ is an adjoint of $f$ if and only if $g' \prl g$. Consequently, if $X$ is irredundant, then $g$ is the unique adjoint of $f$.
\end{lemma}

\begin{proof}
This is immediate from Definition~\ref{def:adjoint}.
\end{proof}

Adjointable maps factor through the relation of equivalence.

\begin{proposition} \label{prop:quotient-map}
Any adjointable map $f \colon X \to Y$ preserves $\prl$. Hence we may define
\[ P(f) \colon P(X) \to P(Y) \komma \class x \mapsto \class{f(x)}. \]
Moreover, if $g \colon Y \to X$ is an adjoint of $f$, then $P(g)$ is the unique adjoint of $P(f)$.
\end{proposition}

\begin{proof}
Again, the assertion is readily checked, or see \cite[Proposition 3.3]{PaVe4}.
\end{proof}

Let $f \colon X \to Y$ be an adjointable map. In the sequel, $f\adj$ will always denote an adjoint of $f$. It is, however, clear from Lemma~\ref{lem:irredundancy-adjoints-1} that there exist further adjoints of $f$ unless $X$ is irredundant.

\begin{lemma} \label{lem:lattice-adjoint}
Let $f \colon X \to Y$ be adjointable.
\begin{itemize}

\item[\rm (i)] For any $A \subseteq X$, we have $f(A\cc) \subseteq f(A)\cc$.

\item[\rm (ii)] For any $A_i \in {\mathsf C}(X)$, $i \in I$, we have $f(\bigvee_i A_i)\cc \;=\; \bigvee_i f(A_i)\cc$.
\end{itemize}
\end{lemma}

\begin{proof}
See \cite[Lemmas~3.5(ii), ~3.7(ii)]{PaVe4}.
\end{proof}

Assume that $A$ is a subspace of the orthoset $X$. When switching to the irredundant quotients, subsethood might no longer hold: the partition of the orthoset $A$ into its equivalence classes might be coarser than the partition of $A$ into those equivalence classes of $X$ that lie in $A$. A condition involving adjointability helps to prevent this undesired situation.

\begin{lemma} \label{lem:PA-subspace-of-PX}
Let $A$ be a subspace of $X$ and assume that the map $\iota \colon A \to X \komma x \mapsto x$ is adjointable. Then $P(A)$ is a subspace of $P(X)$.
\end{lemma}

\begin{proof}
Let $x \in A$ and $y \in X$. We claim that $x$ and $y$ are equivalent elements of $X$ if and only if $y \in A$ and $x$ and $y$ are equivalent elements of $A$. Indeed, $\{x\}\c = \{y\}\c$ implies $y \in \{y\}\cc = \{x\}\cc \subseteq A$ and $\{x\}\ce{A} = \{y\}\ce{A}$. Conversely, if $x$ and $y$ are equivalent elements of $A$, then, by the adjointability of $\iota$ and Proposition \ref{prop:quotient-map}, $x$ and $y$ are equivalent elements of $X$.

We conclude that $P(A) = \{ \class x \colon x \in A\} \subseteq P(X)$. Moreover, in $P(X)$ we have $P(A) = \{ \class x \colon x \perp A\c \} = \{ \class y \colon y \in A\c \}\c \in {\mathsf C}(P(X))$.
\end{proof}

The {\it kernel} of an adjointable map $f \colon X \to Y$ is the set
\[ \kernel f \;=\; \{ x \in X \colon f(x) = 0 \}. \]
Note that the adjointability implies $f(0) = 0$. We say that $f$ has a {\it zero kernel} if the kernel of $f$ consists of $0$ alone. Furthermore, the {\it image} of $f$ is
\[ \image f \;=\; \{ f(x) \colon x \in X \}. \]
The rank of the subspace $(\image f)\cc$ of $Y$ is called the {\it rank} of $f$.

\begin{lemma} \label{lem:injective-surjective}
Let $f \colon X \to Y$ and $g \colon Y \to X$ be an adjoint pair of maps. Then
\[ \kernel f \;=\; (\image g)\c, \qquad \kernel g \;=\; (\image f)\c. \]
\end{lemma}

\begin{proof}
See \cite[Lemma 3.8(i)]{PaVe4}.
\end{proof}

For an adjointable map $f \colon X \to Y$, we define its {\it zero-kernel restriction} $\zerokernel f$ to be $f$ restricted to the subspace $(\kernel f)\c$ of $X$ and corestricted to the subspace $(\image f)\cc$ of $Y$. That is, $\zerokernel f = f\big|_{(\kernel f)\c}^{(\image f)\cc}$. If $g \colon Y \to X$ is an adjoint of $f$, we have that $\zerokernel f$ and $\zerokernel g$ form an adjoint pair of maps between $(\image g)\cc$ and $(\image f)\cc$.

In the remainder of this section, we consider properties of maps involving the preservation of the orthogonality relation.

A bijection $f \colon X \to Y$ such that, for any $x_1, x_2 \in X$, $x_1 \perp x_2$ if and only if $f(x_1) \perp f(x_2)$ is called an {\it orthoisomorphism}. An orthoisomorphism of an orthoset with itself is called an {\it orthoautomorphism}.

\begin{proposition} \label{prop:unitary-maps}
Given $f \colon X \to Y$, the following are equivalent:
\begin{itemize}

\item[\rm (a)] $f$ is an orthoisomorphism;

\item[\rm (b)] $f$ is bijective and $f^{-1}$ is an adjoint of $f$;

\item[\rm (c)] $f$ is bijective and adjointable, and $g \circ f \prl \id_X$ for any adjoint $g$ of $f$.

\end{itemize}
In this case, also $P(f) \colon P(X) \to P(Y)$ is an orthoisomorphism.
\end{proposition}

\begin{proof}
For the equivalences, see \cite[Proposition 4.2]{PaVe4}. The last assertion is readily checked.
\end{proof}

We call a map $f \colon X \to Y$ a {\it partial orthometry} if $f$ possesses an adjoint $g \colon Y \to X$ such that $\zerokernel f$ and $\zerokernel g$ are mutually inverse orthoisomorphisms. We refer in this case to $g$ as a {\it generalised inverse} of $f$. Note that this means the following: $(\image g, \kernel f)$ is a decomposition of $X$, $(\image f, \kernel g)$ is a decomposition of $Y$, and $f$ and $g$ establish mutually inverse orthoisomorphisms between $\image g$ and $\image f$. Moreover, we call an injective partial orthometry an {\it orthometry}.

\begin{proposition} \label{prop:partial-isometries}
Let $f \colon X \to Y$ and $g \colon Y \to X$ be an adjoint pair of maps.
\begin{itemize}

\item[\rm (i)] The following are equivalent:
\begin{itemize}

\item[\rm (a)] $f$ is a partial orthometry and $g$ is a generalised inverse of $f$.

\item[\rm (b)] $\image f$ and $\image g$ are orthoclosed, and $f \circ g \circ f = f$ as well as $g \circ f \circ g = g$.

\end{itemize}
In this case, $P(f) \colon P(X) \to P(Y)$ is likewise a partial orthometry and \linebreak $P(g) \colon P(Y) \to P(X)$ is its generalised inverse. Moreover, $\image P(f) = P(\image f)$ and $\image P(g) = P(\image g)$.

\item[\rm (ii)] The following are equivalent:

\begin{itemize}

\item[\rm (a)] $f$ is an orthometry and $g$ is a generalised inverse of $f$.

\item[\rm (b)] $\image f$ is orthoclosed and $g \circ f = \id_X$.

\end{itemize}
In this case, also $P(f) \colon P(X) \to P(Y)$ is an orthometry. Moreover, $P(\image f) = \image P(f)$.

\end{itemize}
\end{proposition}

\begin{proof}
Ad (i): For the equivalence of (a) and (b), see \cite[Proposition~4.4]{PaVe4}.

Assume (a). As $\image f$ is an orthoclosed subset of $Y$, $\,\image P(f) = \{ \class{f(x)} \colon x \in X \}$ is an orthoclosed subset of $P(Y)$. Similarly we see that also $\image P(g)$ is orthoclosed. Furthermore, $P(f) \circ P(g) \circ P(f) = P(f \circ g \circ f) = P(f)$ and similarly $P(g) \circ P(f) \circ P(g) = P(g)$. Hence $P(f)$ is a partial orthometry and $P(g)$ its generalised inverse.

For any $x \in \image f$ and $y \in Y$, we have $x \perp y$ iff $f(g(x)) \perp y$ iff $x \perp f(g(y))$. Hence the inclusion map $\image f \to Y$ has the adjoint $(f \circ g)|^{\image f}$. By Lemma~\ref{lem:PA-subspace-of-PX}, it follows $P(\image f) \subseteq P(Y)$. Hence $P(\image f) = \{ \class{f(x)} \colon x \in X \} = \image P(f)$. Similarly, we see that $P(\image g) = \image P(g)$.

Ad (ii): For the equivalence, see \cite[Proposition~4.5]{PaVe4}. The additional assertion is seen similarly as in part (i).
\end{proof}

Let $A$ be a subspace of an orthoset $X$. We refer to $\iota \colon A \to X \komma x \mapsto x$ as the {\it inclusion map} of $A$ into $X$. We readily check that $\iota$ is an orthometry if and only if $\iota$ is adjointable \cite[Lemma~4.8(ii)]{PaVe4}. In this case, we call a generalised inverse of $\iota$ a {\it Sasaki map} of $X$ onto $A$. In other words, a map $\sigma \colon X \to A$ is a Sasaki map if $\sigma$ is an adjoint of $\iota$ such that $\sigma|_A = \id_A$.

Let $A$ still be a subspace of $X$. A partial orthometry $p \colon X \to X$ such that $\zerokernel p = \id_A$ is called a {\it projection} of $X$ onto $A$.

\begin{lemma} \label{lem:projections}
Let $p \colon X \to X$. The following are equivalent:
\begin{itemize}

\item[\rm (a)] $p$ is a projection.

\item[\rm (b)] There is a Sasaki map $\sigma$ of $X$ onto a subspace $A$ of $X$ such that $p = \iota \circ \sigma$, where $\iota \colon A \to X$ is the inclusion map.

\item[\rm (c)] $p$ is idempotent and self-adjoint, and $\image p$ is orthoclosed.

\end{itemize}
\end{lemma}

\begin{proof}
See \cite[Lemma 4.9]{PaVe4}.
\end{proof}

The following proposition explains how Dacey spaces occur in the present context.

\begin{proposition} \label{prop:adjoints-and-Dacey}
Assume that for any subspace $A$ of $X$, the inclusion map $\iota \colon A \to X$ is adjointable. Then $X$ is a Dacey space. Moreover, the following hold:
\begin{itemize}

\item[\rm (i)] If $X$ is atomistic, then ${\mathsf C}(X)$ is an AC orthomodular lattice.

\item[\rm (ii)] For any subspace $A$, $\,P(A)$ is a subspace of $P(X)$.

\end{itemize}
\end{proposition}

\begin{proof}
See \cite[Lemma 4.10]{PaVe4}. Part (ii) follows from Lemma~\ref{lem:PA-subspace-of-PX}.
\end{proof}

The following lemma will in what follows be particularly important.

\begin{lemma} \label{lem:image-of-partial-isometry}
Let $X$ and $Y$ be atomistic Dacey spaces and let $f \colon X \to Y$ and $g \colon Y \to X$ be an adjoint pair of maps.
\begin{itemize}

\item[\rm (i)] If $f \circ g \circ f \prl f$, then $P(f) \colon P(X) \to P(Y)$ is a partial orthometry and $P(g)$ is its generalised inverse.

\item[\rm (ii)] If $g \circ f \prl \id_X$, then $P(f) \colon P(X) \to P(Y)$ is an orthometry and $P(g)$ is its generalised inverse.

\end{itemize}
\end{lemma}

\begin{proof}
See \cite[Lemma 4.11]{PaVe4}.
\end{proof}

\section{Orthomodular spaces}
\label{sec:orthomodular-spaces}

By Example~\ref{ex:basic-1}, any Hilbert space has the natural structure of an orthoset. We compile in this section the relevant facts about those inner-product spaces that do as well. The material is in large part taken from \cite{PaVe5} and adjusted to our present purposes.

By a \emph{$\ast$-sfield}, we mean a sfield (i.e., a skew field) equipped with an involutory antiautomorphism. Examples of $\ast$-sfields include $\Reals$, $\Complexes$, or $\Quaternions$, equipped with their respective standard antiautomorphisms. We note that the star is part of the notion and not necessarily used otherwise. By default, however, we do denote the involution by $\ast$ (written exponentially).

Let $H$ be a linear space over the $\ast$-sfield $F$. A \emph{Hermitian form} on $H$ is a symmetric sesquilinear form, that is, a map $\herm{\cdot}{\cdot} \colon H \times H \to F$ such that, for any $u, v, w \in H$ and $\alpha, \beta \in F$,
\begin{align*}
& \herm{\alpha u + \beta v}{w} \;=\; \alpha \herm{u}{w} + \beta \herm{v}{w}, \\
& \herm{w}{\alpha u + \beta v} \;=\;
                               \herm{w}{u} \alpha\inv + \herm{w}{v} \beta\inv, \\
& \herm{u}{v} \;=\; \herm{v}{u}\inv.
\end{align*}
We additionally assume that Hermitian forms are \emph{anisotropic}, that is, $\herm{u}{u} = 0$ implies $u = 0$. Endowed with an (anisotropic) Hermitian form, $H$ is referred to as a \emph{Hermitian space}.

We call vectors $u, v \in H$ {\it orthogonal} if $\herm u v = 0$ and we write $u \perp v$ in this case. We tacitly assume Hermitian spaces to contain at most countably many pairwise orthogonal vectors. The cardinality of any maximal set of pairwise orthogonal non-zero vectors is unique and referred to as the {\it dimension} of $H$.

It is clear that Examples~\ref{ex:basic-1} and~\ref{ex:basic-2} apply to this more general context as well.

\begin{example} \label{ex:Hermitian-space}
Let $H$ be a Hermitian space. According to our definition, the Hermitian form is symmetric as well as anisotropic. It follows that $H$, equipped with the orthogonality relation $\perp$ and the zero vector $0$, is an orthoset. Note that the rank of $H$ as an orthoset coincides with the dimension of $H$ as a Hermitian space.

The irredundant quotient $P(H)$ of $H$ consists of the non-zero multiples of the non-zero vectors as well as the zero subspace, cf.\ {\rm (\ref{fml:basic-2})}. Indeed, with the notation of Example~\ref{ex:basic-1}, we have for any $u \in H\withoutzero$ that $v \in \class u$ iff $\{v\}\cc = \{u\}\cc$ iff $\lin v = \lin u$ iff $v \in \lin u \setminus \{0\}$. Hence $\class u = \lin u \setminus \{0\}$.
\end{example}

By a {\it subspace} $S$ of a Hermitian space $H$, we mean an orthoclosed subset of $H$, equipped with the restriction of the linear structure and the Hermitian form. Clearly, $S$ becomes in this way a Hermitian space as well. We write $\lin{u_1, \ldots, u_k}$ for the linear span of vectors $u_1, \ldots, u_k \in H$ and we note that this is a subspace of $H$. Indeed, we have $\lin{u_1, \ldots, u_k} = \{u_1, \ldots, u_k\}\cc$.

Note that our notion of a subspace for Hermitian spaces is consistent with the notion for orthosets. However, the decomposability of a Hermitian space $H$ into subspaces may not be equivalent to the decomposability of $H$ as an orthoset. In this paper we are exclusively interested in Hermitian spaces for which the meaning of decomposability coincides in both cases.

Given two Hermitian spaces $H_1$ and $H_2$ over the same $\ast$-sfield $F$, we define their direct sum $H_1 \oplus H_2$ in the obvious way. We consider $H_1$ and $H_2$ as subspaces of $H_1 \oplus H_2$. Note that $(H_1, H_2)$ is a decomposition of $H_1 \oplus H_2$ seen as an orthoset, because $H_1$ and $H_2$ are mutual orthocomplements.

Conversely, let $(S,S\c)$ be a decomposition of a Hermitian space $H$ seen as an orthoset. Note that then $S$ and $S\c$ are subspaces of $H$ such that $S \cap S\c = \{0\}$. In case when $H = S + S\c$, that is, if $H$ is the direct sum of $S$ and $S\c$, we call $S$ {\it splitting}. An {\it orthomodular space} is a Hermitian space $H$ such that every subspace of $H$ is splitting. This means that the decompositions of the orthoset $H$ coincide with the decompositions of the Hermitian space $H$ into direct sums of subspaces.

We finally require orthomodular spaces to fulfil a homogeneity condition. We call $\herm u u$ the {\it length} of a vector $u$ of an orthomodular space $H$. A {\it unit vector} is a vector of length $1$. We call $H$ {\it uniform} if every $1$-dimensional subspace of $H$ contains a unit vector.

\begin{theorem} \label{thm:Soler}
\begin{itemize}

\item[\rm (i)] Any finite-dimensional Hermitian space is orthomodular.

\item[\rm (ii)] A Hilbert space over $\Reals$, $\Complexes$, or $\Quaternions$ is a uniform orthomodular space.

\item[\rm (iii)] Any infinite-dimensional uniform orthomodular space is a Hilbert space over $\Reals$, $\Complexes$, or $\Quaternions$.

\item[\rm (iv)] Any orthomodular space over $\Reals$, $\Complexes$, or $\Quaternions$ is a Hilbert space.

\end{itemize}
\end{theorem}

\begin{proof}
Ad (i): Let $S$ be a subspace of a finitely-dimensional Hermitian space $H$. Then $S$ possesses an orthogonal basis, which can be extended to an orthogonal basis of $H$. Hence $S$ is splitting.

Ad (ii): Any orthoclosed subset $S$ of a Hilbert space $H$ is closed w.r.t.\ the metric on $H$ and from the completeness of the metric it follows that $S$ is a splitting subspace. Moreover, for any $u \in H \setminus \{0\}$, $\herm u u$ is a positive real and $\frac 1 {\sqrt{\herm u u}} u$ has length $1$.

Ad (iii): This is Sol\` er's Theorem \cite{Sol}.

Ad (iv): By the Amemiya-Araki Theorem, any orthomodular space over $\Complexes$ is complete and hence a Hilbert space \cite{AmAr}. The proof works for $\Reals$ and $\Quaternions$ as well.
\end{proof}

A $\ast$-sfield $F$ is called {\it Pythagorean} if, for any $\alpha, \beta \in F$, there is a $\gamma \in F$ such that $\alpha \alpha\inv + \beta \beta\inv = \gamma \gamma\inv$. Furthermore, $F$ is called {\it formally real} if, for any $\alpha_1, \ldots, \alpha_k \in F$, $k \geq 1$, we have that $\alpha_1 {\alpha_1}\inv + \ldots + \alpha_k {\alpha_k}\inv = 0$ implies $\alpha_1 = \ldots = \alpha_k = 0$.

\begin{lemma} \label{lem:sfield-of-transitive-orthomodular-spaces}
Let $H$ be a uniform orthomodular space over a $\ast$-sfield $F$ and assume that $H$ is at least $2$-dimensional. Then $F$ is Pythagorean and formally real.

Conversely, let $F$ be a Pythagorean and formally real $\ast$-sfield and let $n \in \Naturals \setminus \{0\}$. Then $F^n$, equipped with the standard inner product, is a uniform orthomodular space.
\end{lemma}

\begin{proof}
See \cite[Lemma~3.1,~3.2,~3.3]{KPV}.
\end{proof}

\begin{lemma} \label{lem:direct-sum-of-orthomodular-spaces}
\begin{itemize}

\item[\rm (i)] Let $H_1$ and $H_2$ be uniform orthomodular spaces over the same Pythagore\-an $\ast$-sfield $F$. Then $H_1 \oplus H_2$ is likewise a uniform orthomodular space.

\item[\rm (ii)] Let $H$ be a uniform orthomodular space. Then any subspace of $H$ is uniform and orthomodular as well.

\end{itemize}
\end{lemma}

\begin{proof}
Ad (i): If $H_1$ or $H_2$ is infinite-dimensional, then $H_1$ and $H_2$ are, by Theorem~\ref{thm:Soler}(iii), Hilbert spaces and the assertion follows by Theorem~\ref{thm:Soler}(ii).

Assume that $H_1$ and $H_2$ are finite-dimen\-sional. Then $H_1 \oplus H_2$ is finite-dimensional as well and hence, by Theorem~\ref{thm:Soler}(i), orthomodular. Moreover, let $x \in (H_1 \oplus H_2) \setminus \{0\}$. Then $x = \alpha_1 e_1 + \alpha_2 e_2$ for unit vectors $e_1 \in H_1$ and $e_2 \in H_2$. As $F$ is assumed to be Pythagorean, there is a $\gamma \in F$ such that $\alpha_1 \alpha_1\inv + \alpha_2 \alpha_2\inv = \gamma \gamma\inv$. Then $\gamma^{-1} x$ is a unit vector.

Ad (ii): Let $S$ be a subspace of $H$. Obviously, $S$ is uniform. For the orthomodularity of $S$, see \cite[Lemma~1]{Sol}.
\end{proof}

Let $U$ be a unitary operator acting on a Hilbert space $H$. By a {\it strict square root} of $U$ we mean a further unitary operator $V$ such that $V^2 = U$ and for any closed subspace $S$ of $H$, $S$ is reducing for $U$ if and only if $S$ is reducing for $V$.

The existence of strict square roots is a criterion that allows us to distinguish the complex Hilbert spaces from those over the reals or quaternions.

\begin{lemma} \label{lem:strict-square-root-and-C}
Let $H$ be an at least $2$-dimensional Hilbert space over $F \in \{ \Reals, \Complexes, \Quaternions\}$. Then the following are equivalent:
\begin{itemize}

\item[\rm (a)] Every unitary map on $H$ possesses a strict square root.

\item[\rm (b)] $-\id_H$ possesses a strict square root.

\item[\rm (c)] $F = \Complexes$.

\end{itemize}
\end{lemma}

\begin{proof}
(a) $\Rightarrow$ (b): This is trivial.

(b) $\Rightarrow$ (c): Assume that $V$ is a strict square root of $U = -\id_H$. As all $1$-dimensional subspaces are invariant for $U$, so they are for $V$ and hence there is an $\alpha \in F$ such that $V = \alpha \, \id_H$. But then $\alpha$ is an element of the centre of $F$ such that $\alpha^2 = -1$. We conclude that $F = \Complexes$.

(c) $\Rightarrow$ (a): Let $U$ be an arbitrary unitary operator on the complex Hilbert space $H$. Then $\spectrum U$, the spectrum of $U$, is contained in the unit circle. Let $h \colon \spectrum U \to \Complexes$ be defined by $h(e^{it}) = e^{\frac{it}2}$, where $0 \leq t < 2\pi$. Using Borel function calculus, we put $V = h(U)$. Then $h(U)^2 = h^2(U) = U$. Moreover, if a projection $P$ commutes with $h(U)$, then clearly also with $U$. Conversely, $h(U)$ is in the double commutant of $U$. Hence, if $P$ commutes with $U$, then also with $h(U)$.
\end{proof}

Lattice-theoretically, orthomodular spaces of dimension $\geq 4$ can be described as follows. We note that the assumption on the lattice to fulfil the covering property can be replaced with a weaker condition \cite[Theorem~3.10]{PaVe3}.

We call an ortholattice $L$ {\it irreducible} if $L$ is not directly decomposable.

\begin{theorem} \label{thm:orthomodular-space-as-orthoset}
Let $L$ be a complete, irreducible, AC orthomodular lattice of length $\geq 4$. Then there is an orthomodular space $H$ such that $L$ is isomorphic to ${\mathsf C}(H)$.
\end{theorem}

\begin{proof}
See~\cite[Theorem~(34.5)]{MaMa}.
\end{proof}

\section{Adjointable linear maps}
\label{sec:adjointable-linear-maps}

We now discuss the maps between orthomodular spaces that will be relevant for us. We fix an orthomodular space $H_1$ over the $\ast$-sfield $F_1$ and an orthomodular space $H_2$ over the $\ast$-sfield $F_2$. We denote the Hermitian forms by $\herm{\cdot}{\cdot}_1$ and $\herm{\cdot}{\cdot}_2$, respectively.

A map $\phi \colon H_1 \to H_2$ is called {\it quasilinear} if $\phi$ is additive and there is an isomorphism $\sigma \colon F_1 \to F_2$ such that $\phi(\alpha x) = \alpha^\sigma \phi(x)$ for any $x \in H_1$ and $\alpha \in F_1$. In case when $F_1$ and $F_2$ coincide and $\sigma$ is the identity, $\phi$ is called {\it linear}. The image of $\phi$ is denoted by $\image \phi$. In the finite-dimensional case, $\image \phi$ is a subspace of $H_2$ and we refer to its dimension as the {\it rank} of $\phi$.

Any quasilinear map $\phi \colon H_1 \to H_2$ may also be regarded as a map between orthosets. Seen as such, $\phi$ is compatible with the equivalence relation $\prl$ and hence induces the map
\begin{equation} \label{fml:induced-map}
P(\phi) \colon P(H_1) \to P(H_2) \komma \class u \mapsto \class{\phi(u)}.
\end{equation}
We are primarily interested in the case that $\phi$ and $P(\phi)$ are adjointable. Quasilinear maps between orthomodular spaces, however, do not have in general this property. We will hence restrict our considerations to a suitably restricted type of maps. We will consider the relevant notion for linear maps only.

Let $\phi \colon H_1 \to H_2$ be linear. A map $\phi\adj \colon H_2 \to H_1$ such that
\begin{equation} \label{fml:adjoint-of-linear-map}
\herm{\phi(u)}{v} \;=\; \herm{u}{\phi\adj(v)}
\quad\text{for any $u \in H_1$ and $v \in H_2$}
\end{equation}
is called the {\it linear adjoint} of $\phi$. In the finite-dimensional case, any linear map possesses a linear adjoint. A linear map between Hilbert spaces has a linear adjoint if and only if it is bounded. If a linear adjoint exists it is clearly unique. 

It turns out that linear adjoints are adjoints in the sense of Definition~\ref{def:adjoint} and that the adjointability of a linear map between orthomodular spaces has an unambiguous meaning.

\begin{theorem} \label{thm:adjoint-of-map-induced-by-linear-map}
Let $\phi \colon H_1 \to H_2$ be linear. The following are equivalent:
\begin{itemize}

\item[\rm (a)] $\phi$ possesses a linear adjoint.

\item[\rm (b)] $\phi$, seen as a map between orthosets, is adjointable.

\item[\rm (c)] $P(\phi)$ is adjointable.
\end{itemize}
In this case, the linear adjoint $\phi\adj$ is an adjoint of $\phi$ seen as a map between orthosets. Consequently, the unique adjoint of $P(\phi)$ is $P(\phi\adj)$.
\end{theorem}

\begin{proof}
The equivalence of (a) and (c) holds by \cite[Theorem 3.23]{PaVe5} and the equivalence of (b) and (c) holds by \cite[Proposition 3.3]{PaVe4}. The last assertions follow from (\ref{fml:adjoint-of-linear-map}) and Proposition~\ref{prop:quotient-map}.
\end{proof}

We note that the adjoints of decomposable maps are likewise decomposable.

\begin{lemma} \label{lem:linear-adjoints-of-sums}
Assume that $H_1 = R_1 \oplus S_1$ and $H_2 = R_2 \oplus S_2$. Let $\phi \colon R_1 \to R_2$ and $\psi \colon S_1 \to S_2$ be linear maps. Then $\phi \oplus \psi$ possesses a linear adjoint if and only if so do $\phi$ and $\psi$. In this case, $(\phi \oplus \psi)\adj = \phi\adj \oplus \psi\adj$.
\end{lemma}

\begin{proof}
Assume that $\chi = \phi \oplus \psi$ has a linear adjoint. Then, for any $u \in S_1$ and $v \in R_2$, we have $\herm{u}{\chi\adj(v)}_1 = \herm{\chi(u)}{v}_2 = \herm{\psi(u)}{v}_2 = 0$ and hence $\chi\adj(R_2) \subseteq R_1$. Similarly, $\chi\adj(S_2) \subseteq S_1$. We conclude that $\chi\adj|_{R_2}^{R_1}$ is an adjoint of $\phi$ and $\chi\adj|_{S_2}^{S_1}$ is an adjoint of $\psi$. This shows the ``only if'' part. The ``if'' part and the last assertion are easily seen.
\end{proof}

The following theorem could be shortened to the slogan ``adjointability implies quasilinearity''. Indeed, it is a consequence of Faure and Fr\" olicher's version of the Fundamental Theorem of Projective Geometry \cite[Chapter 10]{FaFr} that any adjointable map from $P(H_1)$ to $P(H_2)$ whose image is not contained in a $2$-dimensional subspace is of the form $P(\phi)$ for some quasilinear map $\phi$.

\begin{theorem} \label{thm:adjointable-quasilinear}
Let $f \colon P(H_1) \to P(H_2)$ be an adjointable map of rank $\geq 3$. Then $f$ is induced by a quasilinear map from $H_1$ to $H_2$.
\end{theorem}

\begin{proof}
See \cite[Theorem 3.21]{PaVe5}.
\end{proof}

Any non-zero multiple of a quasilinear map is quasilinear again and induces the same map between the associated irredundant quotients. Provided that the image is not contained in a $1$-dimensional subspace, there are no further maps with this property.

\begin{lemma} \label{lem:uniqueness-of-inducing-quasilinear-map}
Let $\phi \colon H_1 \to H_2$ be a quasilinear map of rank $\geq 2$. Then the quasilinear maps $\psi \colon H_1 \to H_2$ such that $P(\phi) = P(\psi)$ are exactly those of the form $\psi = \kappa \phi$ for some $\kappa \in F_2 \setminus \{0\}$.
\end{lemma}

\begin{proof}
See \cite[Proposition 6.3.6]{FaFr}.
\end{proof}

We are interested in the situation when the quasilinear map inducing an adjointable map $f \colon P(H_1) \to P(H_2)$ can actually be chosen linear. The following corollary is based on the assumption that $H_1$ and $H_2$ possess a common subspace $S$ on which $f$ is the identity. It is in this case understood that the $\ast$-sfields $F_1$ and $F_2$ are the same and that the Hermitian forms coincide on $S$.

\begin{corollary} \label{cor:Fundamental-Theorem}
Assume that $H_1$ and $H_2$ possess a common at least $3$-dimensional subspace $S$. Let $f \colon P(H_1) \to P(H_2)$ be adjointable and such that $f|_{P(S)} = \id_{P(S)}$. Then there is a unique linear map $\phi$ such that $f = P(\phi)$ and $\phi|_S = \id_S$.
\end{corollary}

\begin{proof}
By Theorem~\ref{thm:adjointable-quasilinear}, $f = P(\psi)$ for some quasilinear map $\psi \colon H_1 \to H_2$. Moreover, $P(\psi|_S)$ is the identity on $P(S)$, hence, by Lemma~\ref{lem:uniqueness-of-inducing-quasilinear-map}, $\psi|_S = \kappa \, \id_S$ for some $\kappa \neq 0$. We conclude that $\phi = \kappa^{-1} \psi$ is the map with the indicated properties.
\end{proof}

We call $\phi \colon H_1 \to H_2$ {\it quasiunitary} if $\phi$ is a bijective quasilinear map and there is a $\lambda \in F_2 \setminus \{0\}$ such that, for any $u, v \in H_1$,
\begin{equation} \label{fml:quasiunitary}
\herm{\phi(u)}{\phi(v)}_2 \;=\; {\herm{u}{v}_1}^\sigma \; \lambda,
\end{equation}
where $\sigma$ is the sfield isomorphism associated with $\phi$. If in this case $F_1 = F_2$, $\,\sigma$ is the identity, and $\lambda = 1$, we call $\phi$ {\it unitary}. For $H_1$ and $H_2$ to be {\it isomorphic} means that there is a unitary map between $H_1$ and $H_2$.

We note that if $\phi \colon H_1 \to H_2$ is quasiunitary, $F_1 = F_2$, and the restriction of $\phi$ to some non-zero subspace is unitary, then $\phi$ itself is unitary.

Finally, assume again that $F_1 = F_2$. We call a linear map $\phi \colon H_1 \to H_2$ a {\it linear isometry} if $\image \phi$ is a subspace and $\phi|^{\image \phi}$ is unitary.

\begin{lemma} \label{lem:unitary-maps-and-isometries}
Let $\phi \colon H_1 \to H_2$ be an adjointable linear map.
\begin{itemize}

\item[\rm (i)] $\phi$ is unitary if and only if $\phi$ is bijective and $\phi\adj = \phi^{-1}$. In this case, $\phi$ is an orthoisomorphism.

\item[\rm (ii)] $\phi$ is a linear isometry if and only if $\phi\adj \circ \phi = \id_{H_1}$. In this case, $\phi$ is an orthometry.

\end{itemize}
\end{lemma}

\begin{proof}
Elementary.
\end{proof}

Finally, we call a linear endomorphism $\pi$ of an orthomodular space $H$ a {\it projection} if $H$ has a splitting subspace $S$ such that, for any $u \in H$, $\pi(u)$ is the component of $u$ in $S$.

\begin{lemma} \label{lem:projections-in-orthomodular-spaces}
Let $H$ be an orthomodular space and let $\pi \colon H \to H$ be an adjointable linear map. Then $\pi$ is a projection if and only if $\pi$ is idempotent and self-adjoint.
\end{lemma}

\begin{proof}
Elementary.
\end{proof}

We have the following version of Wigner's Theorem.

\begin{theorem} \label{thm:Wigner-iso}
Assume that $H_1$ and $H_2$ are at least $3$-dimensional. Then any orthoisomorphism between $P(H_1)$ and $P(H_2)$ is induced by a quasiunitary map between $H_1$ and $H_2$.
\end{theorem}

\begin{proof}
See \cite[Theorem 4.7]{PaVe5}.
\end{proof}

Also in this case, we have uniqueness up to a non-zero factor.

\begin{lemma} \label{lem:uniqueness-of-inducing-quasiunitary-map}
Assume that $H_1$ and $H_2$ are at least $2$-dimensional and let $\phi \colon H_1 \to H_2$ be quasiunitary. Then the quasiunitary maps $\psi \colon H_1 \to H_2$ such that $P(\phi) = P(\psi)$ are exactly those of the form $\psi = \kappa \phi$ for some $\kappa \in F_2 \setminus \{0\}$.
\end{lemma}

\begin{proof}
We readily check that, for any $\kappa \neq 0$, also $\kappa \phi$ is quasiunitary. Hence the assertion follows from Lemma~\ref{lem:uniqueness-of-inducing-quasilinear-map}.
\end{proof}

\begin{corollary} \label{cor:Wigner-auto}
Let $H$ be an at least $3$-dimensional orthomodular space over $F$ and let $f \colon P(H) \to P(H)$ be an orthoautomorphism. Assume moreover that there is an at least $2$-dimensional subspace $S$ of $H$ such that $f|_{P(S)} = \id_{P(S)}$. Then there is a unique unitary map $\phi \colon H \to H$ inducing $f$ such that $\phi|_S = \id_S$.
\end{corollary}

\begin{proof}
Based on Theorem~\ref{thm:Wigner-iso} and Lemma~\ref{lem:uniqueness-of-inducing-quasiunitary-map}, we may argue as in case of Corollary~\ref{cor:Fundamental-Theorem}.
\end{proof}

In the following sense, orthometries can be represented by isometries.

\begin{lemma} \label{lem:extension-of-representation}
Assume that $H_1$ is at least $3$-dimensional and let $f \colon P(H_1) \to P(H_2)$ be an orthometry. Then there is an orthomodular space $H_2'$ containing $H_1$ as a subspace and an orthoisomorphism $r \colon P(H_2) \to P(H_2')$ such that $r \circ f = P(\iota)$, where $\iota \colon H_1 \to H_2'$ is the inclusion map.
\end{lemma}

\begin{proof}
By assumption, $\image f$ is orthoclosed and hence equal to $P(S)$ for a subspace $S$ of $H_2$. Then $f|^{\image f} \colon P(H_1) \to P(S)$ is an orthoisomorphism and by Theorem~\ref{thm:Wigner-iso} induced by a quasiunitary map $\check\phi \colon H_1 \to S$. Let $\sigma \colon F_1 \to F_2$ be the associated isomorphism of sfields and let $\lambda \in F_2 \setminus \{0\}$ be given according to (\ref{fml:quasiunitary}). We put $\phi \colon H_1 \to H_2 \komma u \mapsto \check\phi(u)$. Then $f = P(\phi)$.

We redefine as follows the linear structure as well as the Hermitian form on $H_2$ and we denote the resulting orthomodular space by $H_2'$. The addition of vectors remains unmodified, but $F_1$ is taken as the scalar $\ast$-sfield, the scalar multiplication being defined by $\alpha \scmu u = \alpha^\sigma u$ for $\alpha \in F_1$ and $u \in H_2'$. The new form is given by $\herm{u}{v}_2' = \big(\herm{u}{v}_2 \, \lambda^{-1}\big)^{\sigma^{-1}}$ for $u, v \in H_2'$.

Let $\tau \colon H_2 \to H_2'$ be the identity and $r = P(\tau)$. Since the orthogonality relation on $H_2'$ is the same as on $H_2$, $r$ is an orthoisomorphism between $P(H_2)$ and $P(H_2')$. We furthermore readily check that $\iota = \tau \circ \phi$ is an isometry from $H_1$ to $H_2'$, and we have $r \circ f = P(\tau \circ \phi) = P(\iota)$. Using $\iota$ to identify $H_1$ with a subspace of $H_2'$, the assertion follows.
\end{proof}

\subsubsection*{Orthomodular spaces and symmetries}

Hilbert spaces might be regarded as highly symmetric and hence our interest is focused on orthomodular spaces that likewise possess a reasonable amount of automorphisms. We call an orthomodular space $H$ {\it transitive} if for any non-zero vectors $u, v \in H$ there is a unitary map $\phi \colon H \to H$ such that $\phi(u) \in \lin v$.

\begin{lemma} \label{lem:transitivity-and-unit-vectors}
Let $H$ be an orthomodular space that contains a unit vector. Then $H$ is transitive if and only if $H$ is uniform.
\end{lemma}

\begin{proof}
The ``only if'' part is clear. To see the ``if'' part, let $u, v \in H$ be two unit vectors. As the subspace $\lin{u,v}$ is splitting, there is a unitary map $\phi$ such that $\phi(u) = v$ and $\phi(w) = w$ if $w \perp \{u,v\}$.
\end{proof}

The following lemma shows that the requirement for an orthomodular space to contain a unit vector is in our context not restrictive \cite{Hol2}.

\begin{lemma} \label{lem:rescaling}
Let $H$ be an orthomodular space. Then there is an orthomodular space $H'$ such that $P(H)$ and $P(H')$ are orthoisomorphic and $H'$ contains a unit vector. If $H$ is transitive, then $H'$ is transitive and consequently uniform.
\end{lemma}

\begin{proof}
Let $z \in H \setminus \{0\}$ and $\lambda = \herm z z$. We define the orthomodular space $H'$ as follows: as a linear space, $H'$ coincides with $H$; the involution of the scalar sfield $F$ is given by $\alpha^{\star'} = \lambda \alpha\inv \lambda^{-1}$, $\alpha \in F$; and the Hermitian form is $\herm{\cdot}{\cdot}' = \herm{\cdot}{\cdot} \lambda^{-1}$. Then $z$ is a unit vector of $H'$.

Clearly, $P(H)$ and $P(H')$ are orthoisomorphic. Moreover, since the unitary groups of $H$ and $H'$ coincide, transitivity of $H$ implies the transitivity of $H'$.
\end{proof}

\section{Dagger categories}
\label{sec:dagger-categories}

It is our aim to describe Hilbert spaces as categories of orthosets. In this context, it is natural to work in the framework of categories that are augmented by an endofunctor called dagger \cite{AbCo,Sel}. We provide in this section a short introduction to the categorical notions and facts that are relevant for us.

A {\it dagger} on a category $\C$ is a functor $\adj \colon \C\op \to \C$ that is involutive and the identity on objects. A category equipped with a dagger is called a {\it dagger category}.

We are interested in the following example.
\begin{example} \label{ex:categories}
Let $F$ be a Pythagorean, formally real $\ast$-sfield. We define the dagger category $\OMS{F}$ as follows. An object of $\OMS{F}$ is a uniform orthomodular space over $F$. The morphisms of $\OMS{F}$ are the adjointable linear maps between orthomodular spaces, and the dagger is the linear adjoint given by {\rm (\ref{fml:adjoint-of-linear-map})}.

Moreover, assume that $F$ is a classical $\star$-sfield, that is, one of \/ $\Reals$, $\Complexes$, or $\Quaternions$ equipped with the respective canonical antiautomorphism. Let $\Hil{F}$ be the dagger category of Hilbert spaces over $F$, bounded linear maps between them, and the linear adjoint as the dagger.

In view of Theorem~\ref{thm:Soler}, we observe the following. If $F$ is not among the classical $\star$-sfields, $\OMS{F}$ contains finite-dimensional spaces only. If $F \in \{ \Reals, \Complexes, \Quaternions \}$, $\OMS{F}$ coincides with $\Hil{F}$.
\end{example}

Let $\C$ be a dagger category. A morphism $f \colon A \to B$ is called a {\it partial isometry} if $f \circ f\adj \circ f = f$, an {\it isometry} or a {\it dagger monomorphism} if $f\adj \circ f = \id_A$, and a {\it dagger isomorphism} if $f\adj \circ f = \id_A$ and $f \circ f\adj = \id_B$. A {\it dagger automorphism} of $A \in \C$ is a dagger isomorphism $f \colon A \to A$.

Let us assume that $\C$ has a {\it zero object} $0$. This means that there is, for any $A$, a unique morphism $0_{0,A} \colon 0 \to A$. We put $0_{A,0} = {0_{0,A}}\adj$ and note that this is the unique morphism $A \to 0$. For arbitrary objects $A$ and $B$, we put $0_{A,B} = 0_{0,B} \circ 0_{A,0}$, called the {\it zero map} from $A$ to $B$. Clearly, ${0_{A,B}}\adj = 0_{B,A}$.

Let $X \in \C$. A {\it dagger subobject} of $X$ is an object $Y$ together with a dagger monomorphism $h \colon Y \to X$. Note that $0_{0,X} \colon 0 \to X$ and $\id_X \colon X \to X$ are dagger subobjects of $X$. We call $X$ {\it dagger simple} if $X$ is non-zero and possesses, up to dagger isomorphism, no further subobjects. In other words, $X$ is dagger simple if $X \neq 0$ and any dagger monomorphism from a non-zero object to $X$ is a dagger isomorphism.

By a {\it dagger biproduct} of $A, B \in \C$, we mean a coproduct
\[ \begin{tikzcd} A \arrow[r, "\iota_A"] & A \oplus B & B \arrow[l, "\iota_B"'] \end{tikzcd} \]
such that $\iota_A, \iota_B$ are dagger monomorphisms and ${\iota_B}\adj \circ \iota_A = 0_{A,B}$. For the remainder of this section, we assume that any pair of objects in $\C$ possesses a dagger biproduct.

The dagger biproduct of a finite number $n$ of objects is defined in the expected way, understood to be $0$ in case $n = 0$. Up to dagger isomorphism, the dagger biproduct is associative and commutative. Moreover, $0$, the zero object, behaves neutrally: $\begin{tikzcd}[cramped] A \arrow[r, "\id_A"] & A & 0 \arrow[l, "\;\;0_{0,A}"'] \end{tikzcd}$ is obviously a dagger biproduct. Hence we may assume $A \oplus 0 = A$.

For morphisms $f \colon A \to C$ and $g \colon B \to D$, we let $f \oplus g$ be the unique morphism making the diagram
\[ \begin{tikzcd}
C \arrow[r, "\iota_C"] & C \oplus D & D \arrow[l, "\iota_D"'] \\
A \arrow[r, "\iota_A"] \arrow[u, "f"] & A \oplus B  \arrow[u, "f \oplus g"] & B \arrow[l, "\iota_B"'] \arrow[u, "g"]
\end{tikzcd} \]
commute. Also the dagger biproduct of morphisms is, up to dagger isomorphisms, commutative and associative, and for any $f \colon A \to B$ we have $f \oplus 0_{0,0} = f$.

According to \cite[Chapter 2]{HeVi}, the dagger commutes with $\oplus$. We provide a short direct proof.

\begin{lemma} \label{lem:f-oplus-isomorphy}
Let $f_1 \colon A_1 \to B_1$ and $f_2 \colon A_2 \to B_2$ be morphisms. Then $(f_1 \oplus f_2)\adj = {f_1}\adj \oplus {f_2}\adj$.
\end{lemma}

\begin{proof}
For $i = 1, 2$, we have $(f_1 \oplus f_2) \circ \iota_{A_i} = \iota_{B_i} \circ f_i$ and $({f_1}\adj \oplus {f_2}\adj) \circ \iota_{B_i} = \iota_{A_i} \circ {f_i}\adj$. We conclude that, for $i, j = 1, 2$,
\begin{align*}
& {\iota_{A_j}}\adj \circ ({f_1}\adj \oplus {f_2}\adj) \circ \iota_{B_i}
\;=\; {\iota_{A_j}}\adj \circ \iota_{A_i} \circ {f_i}\adj
\;=\; \begin{cases} {f_i}\adj & \text{if $i=j$,} \\
0_{B_i,A_j} & \text{if $i \neq j$;} \end{cases} \\
& {\iota_{A_j}}\adj \circ (f_1 \oplus f_2)\adj \circ \iota_{B_i}
\;=\; \big((f_1 \oplus f_2) \circ \iota_{A_j}\big)\adj \circ \iota_{B_i} \\
& \;=\; (\iota_{B_j} \circ f_j)\adj \circ \iota_{B_i}
\;=\; {f_j}\adj \circ {\iota_{B_j}}\adj \circ \iota_{B_i}
\;=\; \begin{cases} {f_i}\adj & \text{if $i=j$,} \\
0_{B_i,A_j} & \text{if $i \neq j$.} \end{cases}
\end{align*}
The coincidence implies the assertion.
\end{proof}

As the following lemma shows, the dagger biproduct of dagger isomorphisms is again a dagger isomorphism. The lemma is a special case of \cite[Theorem~4.6]{HeKa} and again, we include the short direct proof.

\begin{lemma} \label{lem:piecewise-dagger-isomorphy}
Let $f_1 \colon A_1 \to B_1$ and $f_2 \colon A_2 \to B_2$ be dagger isomorphisms. Then $f_1 \oplus f_2 \colon A_1 \oplus A_2 \to B_1 \oplus B_2$ is likewise a dagger isomorphism.
\end{lemma}

\begin{proof}
By definition, $f_1 \oplus f_2$ is the unique isomorphism such that $(f_1 \oplus f_2) \circ \iota_{A_i} = \iota_{B_i} \circ f_i$ for $i = 1, 2$. We have that, for any $i, j = 1, 2$, the following diagram commutes:
\[ \begin{tikzcd}
A_1 \oplus A_2 \arrow[rr, "f_1 \oplus f_2"]
&& B_1 \oplus B_2 \arrow[rr, "(f_1 \oplus f_2)\adj"] \arrow[dr, "{\iota_{B_j}}\adj"']
&& A_1 \oplus A_2 \arrow[dd, "{\iota_{A_j}}\adj"] \\
& B_i \arrow[ur, "\iota_{B_i}"'] && B_j \arrow[dr, "{f_j}\adj"] \\
A_i \arrow[uu, "\iota_{A_i}"] \arrow[ur, "f_i"] \arrow[rr, "\iota_{A_i}"]
&& A_1 \oplus A_2 \arrow[rr, "{\iota_{A_j}}\adj"] && A_j
\end{tikzcd} \]
As $A_1 \oplus A_2$ is a coproduct, ${\iota_{A_j}}\adj \colon A_1 \oplus A_2 \to A_j$ is, for $j = 1, 2$, the unique morphism making the diagram
\[ \begin{tikzcd}
A_1 \oplus A_2 \arrow[r, "f_1 \oplus f_2"] \arrow[drr, "{\iota_{A_j}}\adj" {yshift=-2pt}]
& B_1 \oplus B_2 \arrow[r, "(f_1 \oplus f_2)\adj"]
& A_1 \oplus A_2 \arrow[d, "{\iota_{A_j}}\adj"] \\
A_i \arrow[u, "\iota_{A_i}"] \arrow[r, "\iota_{A_i}"]
& A_1 \oplus A_2 \arrow[r, "{\iota_{A_j}}\adj"'] & A_j
\end{tikzcd} \]
commute for $i = 1, 2$, and since $A_1 \oplus A_2$ is also a product, we conclude that $(f_1 \oplus f_2)\adj \circ (f_1 \oplus f_2) = \id$. Similarly, we see that also $(f_1 \oplus f_2) \circ (f_1 \oplus f_2)\adj = \id$.
\end{proof}

\begin{lemma} \label{lem:isometries-in-dagger-biproduct}
\begin{itemize}

\item[\rm (i)] Let $\begin{tikzcd}[cramped] A \arrow[r, "\iota_A"] & A \oplus B & B \arrow[l, "\iota_B"'] \end{tikzcd}$ be a dagger biproduct. Then $\iota_A = \id_A \oplus 0_{0,B}$ and ${\iota_A}\adj = \id_A \oplus 0_{B,0}$, and similarly for $\iota_B$ and ${\iota_B}\adj$.

\item[\rm (ii)] For any dagger isomorphism $h \colon A \to A'$ and object $B$, $h \oplus 0_{0,B}$ is the coprojection of a dagger biproduct.

\end{itemize}
\end{lemma}

\begin{proof}
Ad (i): From the commutative diagram
\[ \begin{tikzcd}
A \arrow[r, "\iota_A"]
& X & B \arrow[l, "\iota_B"'] \\
A \arrow[u, "\id_A"'] \arrow[r, "\id_A"']
& A \arrow[u, "\iota_A"'] & \; 0 \arrow[l, "0_{0,A}"] \arrow[u, "0_{0,B}"']
\end{tikzcd} \]
we observe that $\iota_A = \id_A \oplus 0_{0,B}$. By Lemma~\ref{lem:f-oplus-isomorphy}, it follows ${\iota_A}\adj = \id_A \oplus 0_{B,0}$.

Ad (ii): Consider $h \oplus 0_{0,B} \colon A \to A' \oplus B$, where $h$ and $B$ are as indicated. Then $h \oplus 0_{0,B} = \iota_{A'} \circ h$, where $\iota_{A'}$ is the coprojection of the dagger biproduct $A' \oplus B$. Moreover, $\begin{tikzcd}[cramped] A \arrow[r, "\iota_{A'} \circ h"] & A' \oplus A & B \arrow[l, "\iota_B"'] \end{tikzcd}$ is likewise a dagger biproduct.
\end{proof}

\begin{lemma} \label{lem:morphism-plus-zero}
Let $f \colon X_1 \to X_1$ be a morphism and $X = X_1 \oplus X_2$. Then $\iota_1 \circ f \circ {\iota_1}\adj = f \oplus 0$.
\end{lemma}

\begin{proof}
We have $\iota_1 \circ f \circ {\iota_1}\adj \circ \iota_1 = \iota_1 \circ f = (f \oplus 0) \circ \iota_1$ and $\iota_1 \circ f \circ {\iota_1}\adj \circ \iota_2 = 0_{X_2,X} = (f \oplus 0) \circ \iota_2$.
\end{proof}

Finally, let us recall the concept of a semiadditive category. This notion applies to categories in general and a discussion can be found in \cite[\S 40]{HeSt}.

A {\it semiadditive structure} on $\C$ is a binary operation $+$ on $\homset(X,Y)$ for each $X, Y \in \C$, subject to the following conditions: (i) equipped with $+$ and $0_{X,Y}$, $\homset(X,Y)$ is a commutative monoid, and (ii) the composition of morphisms is left and right distributive over $+$.

We will write $\Delta_X$ for the diagonal morphism of an object $X$, and $\nabla_X$ for the codiagonal morphism of $X$.

\begin{theorem} \label{thm:semiadditive-structure}
$\C$ possesses a unique semiadditive structure, given by
\begin{equation} \label{fml:semiadditive-structure}
f+g \;=\; \nabla_Y \circ (f \oplus g) \circ \Delta_X
\end{equation}
for morphisms $f, g \colon X \to Y$.
\end{theorem}

\begin{proof}
See \cite[Proposition~40.12]{HeSt}.
\end{proof}

For morphisms $f, g \colon X \to Y$ in $\C$, $f+g$ is thus specified by means of the following diagram, in which every triangle and every square commutes:
\begin{equation} \label{fml:addition}
\begin{tikzcd}
& X \arrow[dl, "\id"', shift right=0.5ex] \arrow[d, dashed, "\Delta_X"] \arrow[dr, "\id", shift left=0.5ex] \\
X \arrow[r, "\iota"', shift right=0.5ex] \arrow[d, "f"]
& X \oplus X \arrow[l, "\iota\adj"', shift right=0.5ex] \arrow[r, "\iota\adj", shift left=0.5ex] \arrow[d, dashed, "f \oplus g"]
& X \arrow[l, "\iota", shift left=0.5ex] \arrow[d, "g"] \\
Y \arrow[r, "\iota"'] \arrow[dr, "\id"', shift right=0.5ex]
& Y \oplus Y \arrow[d, dashed, "\nabla_Y"]
& Y \arrow[l, "\iota"] \arrow[dl, "\id", shift left=0.5ex] \\
& Y
\end{tikzcd}
\end{equation}

\begin{lemma} \label{lem:adjoints-and-plus}
Let $f, g \colon X \to Y$ be morphisms. Then $(f+g)\adj = f\adj + g\adj$.
\end{lemma}

\begin{proof}
We have
\[ (f+g)\adj \;=\; {\Delta_X}\adj \circ (f \oplus g)\adj \circ {\nabla_Y}\adj
\;=\; \nabla_X \circ (f\adj \oplus g\adj) \circ \Delta_Y \;=\; f\adj + g\adj \]
by Lemma~\ref{lem:piecewise-dagger-isomorphy}.
\end{proof}

\begin{lemma} \label{lem:plus-oplus}
For morphisms $f_1 \colon A_1 \to B_1$ and $f_2 \colon A_2 \to B_2$, we have
\[ (f_1 \oplus 0_{A_2,B_2}) + (0_{A_1,B_1} \oplus f_2) \;=\; f_1 \oplus f_2. \]
\end{lemma}

\begin{proof}
We have to show that
\begin{align*}
{i_{B_k}}\adj \circ \left((f_1 \oplus 0_{A_2,B_2}) + (0_{A_1,B_1} \oplus f_2)\right) \circ i_{A_j} \;=\; {i_{B_k}}\adj \circ (f_1 \oplus f_2) \circ i_{A_j}
\end{align*}
for $j, k \in \{1,2\}$. In case $j = k =1$, we compute
\begin{align*}
{i_{B_1}}\adj & \circ (f_1 \oplus f_2) \circ i_{A_1} \;=\; {i_{B_1}}\adj \circ i_{B_1} \circ  f_1 \;=\; f_1
\end{align*}
and, since composition is left and right distributive over $+$,
\begin{align*}
& {i_{B_1}}\adj \circ \left((f_1 \oplus 0_{A_2,B_2}) + (0_{A_1,B_1} \oplus f_2)\right)\circ i_{A_1} \\
& \;=\; {i_{B_1}}\adj \circ \left((f_1 \oplus 0_{A_2,B_2})\circ i_{A_j} + %
(0_{A_1,B_1} \oplus f_2) \circ i_{A_1}\right)\\
& \;=\; {i_{B_1}}\adj \circ \left(i_{B_1} \circ f_1 + i_{B_1} \circ 0_{A_1,B_1} \right) \;=\;
{i_{B_1}}\adj \circ i_{B_1} \circ f_1 + {i_{B_1}}\adj \circ  i_{B_1} \circ 0_{A_1,B_1}\\
& \;=\; f_1 + 0_{A_1,B_1} \;=\; f_1.
\end{align*}
The remaining cases are shown similarly.
\end{proof}

\begin{example} \label{ex:OMSF}
Let $F$ be a Pythagorean, formally real $\ast$-sfield. In the dagger category $\OMS{F}$, we may understand the intuition behind the concepts mentioned in this section.

The zero object of $\OMS{F}$ is the zero space. By Lemma~\ref{lem:direct-sum-of-orthomodular-spaces}, the direct sum of two uniform orthomodular spaces is again a uniform orthomodular space. Hence $\OMS{F}$ has biproducts, a biproduct is given by the direct sum.

The isometries of $\OMS{F}$ are by Lemma~\ref{lem:unitary-maps-and-isometries} the linear isometries and the dagger isomorphisms are the unitary maps. The dagger simple objects of $\OMS{F}$ are the $1$-dimensional spaces.

The zero maps of $\OMS{F}$ are the constant $0$ linear maps. Finally, for linear maps $f, g \colon$ $H_1 \to H_2$ in $\OMS{F}$, $f+g$ is their usual sum.
\end{example}

To conclude the section, we consider functors between dagger categories that take into account the daggers.

By a {\it dagger functor} $\mathsf F$ between dagger categories $\C$ and $\mathcal D$, we mean functor between $\C$ and $\mathcal D$ preserving the dagger of morphisms. We call $\mathsf F$ {\it dagger essentially surjective} if, for any $Y \in \mathcal D$ there is an $X \in \C$ such that $\mathsf F(X)$ is dagger isomorphic to $Y$.

Moreover, following J.~Vicary \cite{Vic}, we call $\mathsf F$ a {\it unitary dagger equivalence} between $\C$ and $\mathcal D$ if there exists a dagger functor $F \colon \C \to \mathcal D$  which is part of an adjoint equivalence of categories, such that the unit and counit natural transformations are dagger isomorphisms at every stage. By \cite[Lemma~V.1]{Vic}, a dagger functor can be made a part of a unitary dagger equivalence if and only if it is full, faithful, and dagger essentially surjective.

\section{Categories of orthosets: three basic hypotheses}
\label{sec:three-basic-hypotheses}

After the lengthy preparations, we are now ready to turn our attention to the actual aims of this work. In a first step, we shall investigate dagger categories of orthosets subjected to three basic assumptions.

Throughout the remainder of this paper, we will denote by $\C$ a dagger category whose objects are orthosets of at most countable rank, whose morphisms are maps between them, and whose dagger assigns to each morphism an adjoint.

We assume that $\C$ fulfils the following conditions:
\begin{itemize}

\item[\rm (H1)] $\C$ has dagger biproducts.

\item[\rm (H2)] Let $X \in \C$ and let $A$ be a subspace of $X$. Then $A$ and $A\c$ as well as the inclusion maps $\iota_A \colon A \to X$ and $\iota_{A\c} \colon A\c \to X$ belong to $\C$ and
\[ \begin{tikzcd} A \arrow[r, "\iota_A"] & X & A\c \arrow[l, "\;\;\iota_{A\c}"'] \end{tikzcd} \]
is a dagger biproduct of $A$ and $A\c$.

\item[\rm (H3)] 
\begin{itemize}

\item[\rm (a)] Any unital orthoset in $\C$ is dagger simple.

\item[\rm (b)] Any two unital orthosets in $\C$ are dagger isomorphic.

\end{itemize}

\end{itemize}

Note that the assumptions are fulfilled if $\C$ consists solely of the zero orthoset and its identity map. We shall disregard this trivial case, that is, we will tacitly assume that $\C$ contains an orthoset of non-zero rank.

\begin{example} \label{ex:OMSF-fulfils-H1-H3}
For any Pythagorean, formally real $\ast$-sfield $F$, the dagger category $\OMS{F}$ fulfils conditions {\rm (H1)}--{\rm (H3)}.

Ad {\rm (H1)}: Biproducts are direct sums, see Example~\ref{ex:OMSF}.

Ad {\rm (H2)}: By definition, an orthomodular space is the direct sum of any pair of complementary subspaces.

Ad {\rm (H3)(a)}: The unital orthosets in $\OMS{F}$ are the $1$-dimensional spaces, which are dagger simple, see Example~\ref{ex:OMSF}.

Ad {\rm (H3)(b)}: Given two $1$-dimensional spaces in $\OMS{F}$, the linear map sending a unit vector to a unit vector is unitary and hence a dagger isomorphism.
\end{example}

\begin{lemma} \label{lem:zeroorthoset}
The zero orthoset $\Zero$ is the zero object of $\C$. For any $X \in \C$, the morphism $\Zero \to X$ is the map sending $0$ to $0$.
\end{lemma}

\begin{proof}
Let $X_0$ be a zero object of $\C$ and assume that $X_0$ has rank $\geq 1$. By (H1), we have the dagger biproduct $\begin{tikzcd}[cramped] X_0 \arrow[r, "\iota_1"] & X_0 \oplus X_0 & X_0 \arrow[l, "\;\;\iota_2"'] \end{tikzcd}$. For any $x \in {X_0}\withoutzero$, we have $\iota_1(x) \perp \iota_2(x)$ because ${\iota_2}\adj(\iota_1(x)) = 0 \perp x$. But this means that $\iota_1$ and $\iota_2$ are distinct morphisms. We conclude that $X_0$ has rank $0$, that is, $X_0 = \Zero$.

Consider a morphism $f \colon \Zero \to X$. Since an adjointable maps sends $0$ to $0$, we have that $f(0) = 0$.
\end{proof}

\begin{lemma} \label{lem:dagger-isos-in-C}
Let $f \colon X \to Y$ be a $\C$-morphism.

\begin{itemize}

\item[\rm (i)] If $f$ is a dagger isomorphism, then $f$ is an orthoisomorphism.

\item[\rm (ii)] Assume that $f$ is an orthoisomorphism. Then $f\adj \circ f \prl \id_X$, and $f$ is a dagger isomorphism if and only if $f\adj \circ f = \id_X$.

\item[\rm (iii)] If $f$ is a dagger monomorphism, then $f$ preserves and reflects $\perp$.

\end{itemize}
\end{lemma}

\begin{proof}
Parts (i) and (ii) are immediate from Proposition~\ref{prop:unitary-maps}.

Ad (iii): $f\adj \circ f = \id_X$ implies that, for any $x_1, x_2 \in X$, $x_1 \perp x_2$ iff $f\adj(f(x_1)) \perp x_2$ iff $f(x_1) \perp f(x_2)$.
\end{proof}

\begin{lemma} \label{lem:inclusion-and-Sasaki-in-C}
Let $A$ be a subspace of $X \in \C$. Then the following hold:
\begin{itemize}

\item[\rm (i)] $A$ as well as the inclusion map $\iota_A \colon A \to X$ are in $\C$. We have $\iota_A = \id_A \oplus 0_{\Zero,A\c}$ and ${\iota_A}\adj = \id_A \oplus 0_{A\c,\Zero}$. Moreover, $\iota_A$ is an orthometry and ${\iota_A}\adj$ is in $\C$ the unique Sasaki map from $X$ onto $A$.

\item[\rm (ii)] Let $p = \iota_A \circ {\iota_A}\adj$. Then $p = \id_A \oplus 0_{A\c,A\c}$. Moreover, $p$ is in $\C$ the unique projection of $X$ onto $A$.

\item[\rm (iii)] $P(A)$ is a subspace of $P(X)$.

\end{itemize}
\end{lemma}

\begin{proof}
Ad (i): By (H2), subspaces and their inclusion maps belong to $\C$, and we have $X = A \oplus A\c$ via the inclusion maps. $\iota_A$ and ${\iota_A}\adj$ may be expressed in the indicated way by Lemma~\ref{lem:isometries-in-dagger-biproduct}.

As $\image \iota_A = A$ is orthoclosed and ${\iota_A}\adj \circ \iota_A = \id_A$, $\,\iota_A$ is by Proposition \ref{prop:partial-isometries}(ii) an orthometry and ${\iota_A}\adj$ is a Sasaki map from $X$ onto $A$. If $\sigma \colon X \to A$ is a further Sasaki map, then, again by Proposition~\ref{prop:partial-isometries}(ii), $\sigma$ is an adjoint of $\iota_A$ such that $\sigma \circ \iota_A = \id_A$. By Lemma~\ref{lem:injective-surjective}, $\kernel \sigma = (\image \iota_A)\c = A\c$ and hence $\sigma \circ \iota_{A\c} = 0_{A\c,A}$. As $X$ is the coproduct of $A$ and $A\c$, it follows $\sigma = {\iota_A}\adj$.

Ad (ii): By part (i) and Lemma~\ref{lem:projections}, $p$ is a projection of $X$ onto $A$. Moreover, $p = (\id_A \oplus 0_{\Zero,A\c}) \circ (\id_A \oplus 0_{A\c,\Zero}) = \id_A \oplus 0_{A\c,A\c}$.

Any further projection $p' \in \C$ from $X$ onto $A$ is, by Lemma~\ref{lem:projections}, of the form $p' = \iota_A \circ \sigma$ for a Sasaki map $\sigma \colon X \to A$. As $\sigma = {\iota_A}\adj \circ p' \in \C$, we have $\sigma = {\iota_A}\adj$ by part (i) and hence $p' = p$.

Ad (iii): As $\iota_A$ is adjointable, this holds by Lemma \ref{lem:PA-subspace-of-PX}.
\end{proof}

\begin{lemma} \label{lem:restriction-corestriction-in-C}
Let $f \colon X \to Y$ be a $\C$-morphism.
\begin{itemize}

\item[\rm (i)] Let $A$ be a subspace of $X$ and let $B$ be a subspace of $Y$ such that $\image f \subseteq B$. Then $f|_A^B$ is likewise a morphism.

If, in addition, $\kernel f \supseteq A\c$, then $(f|_A^B)\adj = f\adj|_B^A$.

\item[\rm (ii)] The zero-kernel restriction $\zerokernel f$ is a $\C$-morphism.

\end{itemize}
\end{lemma}

\begin{proof}
Ad (i): Let $\iota_A \colon A \to X$ and $\iota_B \colon B \to Y$ be the inclusion maps. By Lem\-ma~\ref{lem:inclusion-and-Sasaki-in-C}(i), $f|_A^B = {\iota_B}\adj \circ f \circ \iota_A$. If $\kernel f \supseteq A\c$, we have $\image f\adj \subseteq A$ and hence $f\adj|_B^A = {\iota_A}\adj \circ f\adj \circ \iota_B = (f|_A^B)\adj$.

Ad (ii): This is clear from part (i).
\end{proof}

\begin{lemma} \label{lem:C-consists-of-Dacey-spaces}
Any $X \in \C$ is an atomistic Dacey space and ${\mathsf C}(X)$ is an AC orthomodular lattice.
\end{lemma}

\begin{proof}
By Lemma~\ref{lem:inclusion-and-Sasaki-in-C}(i), all the inclusion maps of the subspaces into $X$ are adjointable. Hence, by Proposition~\ref{prop:adjoints-and-Dacey}, $X$ is a Dacey space.

Assume, for sake of contradiction, that $X$ is not atomistic. Then there are $x, y \in X\withoutzero$ such that $\{y\}\cc \subsetneq \{x\}\cc$. By orthomodularity, $\{y\}\cc$ is a subspace of the orthoset $\{x\}\cc \in \C$. By (H2), the inclusion map $\iota \colon \{y\}\cc \to \{x\}\cc$ is a dagger monomorphism. By (H3)(a) $\{x\}\cc$ is dagger simple, hence $\iota$ is a dagger isomorphism and in particular bijective. But this means $\{y\}\cc = \{x\}\cc$, a contradiction.

It now follows from Proposition~\ref{prop:adjoints-and-Dacey}(i) that ${\mathsf C}(X)$ is AC.
\end{proof}

\begin{lemma} \label{lem:extended-dagger-biproduct}
Let $X \in \C$.
\begin{itemize}

\item[\rm (i)] Let $(A_1, \ldots, A_k)$ be a decomposition of $X$. Then $X$, together with the inclusion maps, is the dagger biproduct of $A_1, \ldots, A_k$.

\item[\rm (ii)] Assume that $X = A_1 \oplus \ldots \oplus A_k$. Then the coprojections $\iota_{A_1} \colon A_1 \to X, \ldots,$ $\iota_{A_k} \colon A_k \to X$ are orthometries, and $(\image \iota_{A_1}, \ldots, \image \iota_{A_k})$ is a decomposition of $X$.

\end{itemize}
\end{lemma}

\begin{proof}
Ad (i): By Lemma~\ref{lem:C-consists-of-Dacey-spaces}, ${\mathsf C}(X)$ is orthomodular. Hence, for every $i = 2, \ldots, k$, $(A_1 \vee \ldots \vee A_{i-1}, \; A_i)$ is a decomposition of the subspace $A_1 \vee \ldots \vee A_i$. Thus the assertion follows from (H2) by an inductive argument.

Ad (ii): Let $x \in A_i$ and $y \in A_j$, where $i \neq j$. Then ${\iota_{A_i}}\adj(\iota_{A_j}(y)) = 0 \perp x$ and hence $\iota_{A_i}(x) \perp \iota_{A_j}(y)$. We conclude that $\image \iota_{A_1}, \ldots, \image \iota_{A_k}$ are pairwise orthogonal.

Let $Y = (\image \iota_{A_1})\cc \vee \ldots \vee (\image \iota_{A_k})\cc$. We claim that $Y = X$. Indeed, by (H2), $\begin{tikzcd}[cramped] Y \arrow[r, "\iota_Y"] & X & Y\c \arrow[l, "\iota_{Y\c}"'] \end{tikzcd}$ is a dagger biproduct and for any $f \colon Y\c \to Y\c$, the kernel of $0_{Y,Y} \oplus f \colon X \to X$ contains the subspace $Y$. We conclude that $(0_{Y,Y} \oplus f) \circ \iota_{A_i} = 0_{A_i,X}$ for any $i = 1, \ldots, k$ and any $f$, including the cases $f = 0_{Y\c,Y\c}$ and $f = \id_{Y\c}$.
\[ \begin{tikzcd}
Y \arrow[r, "\iota_Y"]
& X & Y\c \arrow[l, "\iota_{Y\c}"'] \\
Y \arrow[u, "0_{Y,Y}"'] \arrow[r, "\iota_Y"']
& X \arrow[u, "0_{Y,Y} \oplus f"'] & Y\c \arrow[l, "\iota_{Y\c}"] \arrow[u, "f"']
\end{tikzcd}
\quad\quad
\begin{tikzcd}
& X \\
A_i \arrow[ur, "0_{A_i,X}" {yshift=-2pt}]  \arrow[r, "\iota_{A_i}"']
& X \arrow[u, "\exists!", "0_{Y,Y} \oplus f"', dashed, {yshift=-2pt}]
\end{tikzcd} \]
By the definition of a biproduct, $0_{Y,Y} \oplus f$ is unique. Hence $0_{Y\c,Y\c} = \id_{Y\c}$, that is, $Y\c = \{0\}$ and $Y = X$.

It remains to show that $\image \iota_{A_1}, \ldots, \image \iota_{A_k}$ are orthoclosed. Then it will follow by Lemma \ref{lem:Dacey-space} that $(\image \iota_{A_1}, \ldots, \image \iota_{A_k})$ is a decomposition of $X$. It will furthermore follow by Proposition~\ref{prop:partial-isometries}(ii), $\iota_{A_1}, \ldots, \iota_{A_k}$ are orthometries.

Let $1 \leq i \leq k$. As $\iota_{A_i}$ is a dagger monomorphism, we have by Lemma~\ref{lem:image-of-partial-isometry} that $\image P(\iota_{A_i})$ is an orthoclosed subset of $P(X)$. Let $x \in {A_i}\withoutzero$. We have to show that $\iota_{A_i}$ maps $\class x$ onto $\class{\iota_{A_i}(x)}$.

Let $\iota_{\{x\}\cc} \colon \{x\}\cc \to A_i$ and $\iota_{\{\iota_{A_i}(x)\}\cc} \colon \{\iota_{A_i}(x)\}\cc \to X$ be the inclusion maps. Then $\iota_{A_i}\big|_{\{x\}\cc}^{\{\iota_{A_i}(x)\}\cc} = {\iota_{\{\iota_{A_i}(x)\}\cc}}\adj \circ \iota_{A_i} \circ \iota_{\{x\}\cc}$ is a dagger monomorphism because, by Lemma~\ref{lem:inclusion-and-Sasaki-in-C}(i),
\[ \begin{split}
& ({\iota_{\{\iota_{A_i}(x)\}\cc}}\adj \circ \iota_{A_i} \circ \iota_{\{x\}\cc})\adj \circ {\iota_{\{\iota_{A_i}(x)\}\cc}}\adj \circ \iota_{A_i} \circ \iota_{\{x\}\cc} \\
& \;=\; {\iota_{\{x\}\cc}}\adj \circ {\iota_{A_i}}\adj \circ \iota_{\{\iota_{A_i}(x)\}\cc} \circ {\iota_{\{\iota_{A_i}(x)\}\cc}}\adj \circ \iota_{A_i} \circ \iota_{\{x\}\cc} \\
& \;=\; {\iota_{\{x\}\cc}}\adj \circ {\iota_{A_i}}\adj \circ \iota_{A_i} \circ \iota_{\{x\}\cc}
\;=\; {\iota_{\{x\}\cc}}\adj \circ \iota_{\{x\}\cc}
\;=\; \id_{\{x\}\cc}.
\end{split} \]
By (H3)(a), $\iota_{A_i}\big|_{\{x\}\cc}^{\{\iota_{A_i}(x)\}\cc}$ is bijective. By Lemmas~\ref{lem:C-consists-of-Dacey-spaces} and~\ref{lem:subspaces-of-Dacey-spaces}(iii), both $A_i$ and $X$ are atomistic and by Lemma~\ref{lem:unital}(ii), $\{x\}\cc = \class x \cup \{0\}$ and $\{\iota_{A_i}(x)\}\cc = \class{\iota_{A_i}(x)}\cup \{0\}$. We conclude that $\iota_{A_i}$ maps $\class x$ bijectively to $\class{\iota_{A_i}(x)}$.
\end{proof}

For any dagger biproduct $X = A_1 \oplus \ldots \oplus A_k$, we may by Lemma~\ref{lem:extended-dagger-biproduct} identify the orthosets $A_1, \ldots, A_k$ with subspaces of $X$. This is what we will usually do from now on. That is, we usually consider dagger biproducts as decompositions.

$(A_1, \ldots, A_k)$ being a decomposition of $X$, let $1 \leq i_1 < \ldots < i_l \leq k$. We note that we may then identify the subspace $A_{i_1} \vee \ldots \vee A_{i_l}$ of $X$ with the dagger biproduct $A_{i_1} \oplus \ldots \oplus A_{i_l}$. Indeed, $(A_{i_1}, \ldots, A_{i_l})$ is a decomposition of $A_{i_1} \vee \ldots \vee A_{i_l}$ by criterion (a) of Lemma~\ref{lem:Dacey-space}. Hence $A_{i_1} \vee \ldots \vee A_{i_l} = A_{i_1} \oplus \ldots \oplus A_{i_l}$ via the inclusion maps.

\begin{lemma} \label{lem:decomposing-a-morphism}
Let $(A_1, \ldots, A_k)$ be a decomposition of $X \in \C$ and let $(B_1, \ldots, B_k)$ be a decomposition of \/ $Y \in \C$. Let $f \colon X \to Y$ be a $\C$-morphism such that $f(A_i) \subseteq B_i$ for each $i = 1, \ldots, k$. Then we have:
\begin{itemize}

\item[\rm (i)] $f = f|_{A_1}^{B_1} \oplus \ldots \oplus f|_{A_k}^{B_k}$.

\item[\rm (ii)] For $1 \leq i_1 < \ldots < i_l \leq k$, we have $f(A_{i_1} \oplus \ldots \oplus A_{i_l}) \subseteq B_{i_1} \oplus \ldots \oplus B_{i_l}$.

\end{itemize} 
\end{lemma}

\begin{proof}
Ad (i): By Lemma~\ref{lem:extended-dagger-biproduct}(i), $X = A_1 \oplus \ldots \oplus A_k$ and $Y = B_1 \oplus \ldots \oplus B_k$. By Lemma \ref{lem:inclusion-and-Sasaki-in-C}(i), $f \circ \iota_{A_i} = \iota_{B_i} \circ {\iota_{B_i}}\adj \circ f \circ \iota_{A_i} = \iota_{B_i} \circ f|_{A_i}^{B_i}$. The assertion follows.

Ad (ii): We have $f(A_{i_1} \vee \ldots \vee A_{i_l}) \subseteq f(A_{i_1})\cc \vee \ldots \vee f(A_{i_l})\cc \subseteq B_{i_1} \vee \ldots \vee B_{i_l}$ by Lemma~\ref{lem:lattice-adjoint}(ii). In view of our preceding remarks, the claim follows.
\end{proof}

We next show that any orthoset $X \in \C$ that has finite rank can be decomposed into smallest non-zero constituents: $X$ can be built up from unital orthosets by means of the dagger biproduct. It will follow that any object in $\C$ is up to dagger isomorphism uniquely characterised by its rank if the latter is finite.

\begin{lemma} \label{lem:nI}
\begin{itemize}

\item[\rm (i)] An orthoset $X \in \C$ is a singleton if and only if $X$ is unital.

\item[\rm (ii)] Let $1 \leq n < \aleph_0$. An orthoset $X \in \C$ has rank $n$ if and only if $X$ is the dagger biproduct of $n$ singletons.

\end{itemize}
\end{lemma}

\begin{proof}
Ad (i): By Lemma~\ref{lem:C-consists-of-Dacey-spaces}, $X$ is atomistic. Hence the assertion holds by Lem\-ma~\ref{lem:unital}(iii).

Ad (ii): Let $X$ have rank $n \geq 1$. Let $x_1, \ldots, x_n$ be a $\perp$-set in $X$. As $X$ is a Dacey space, $(\{x_1\}\cc, \ldots, \{x_n\}\cc)$ is a decomposition of $X$. By Lemma~\ref{lem:extended-dagger-biproduct}(i), Lemma~\ref{lem:unital}(i), and part (i), it follows that $X$ is the dagger biproduct of $n$ singletons.

Conversely, assume that $X = I_1 \oplus \ldots \oplus I_n$ for singletons $I_1 = \{x_1\}\cc, \ldots, I_n = \{x_n\}\cc$. By Lemma~\ref{lem:extended-dagger-biproduct}(ii), we have that $\image \iota_{I_i} = \{\iota_{I_i}(x_i)\}\cc$ for $i = 1, \ldots, n$ and that $X$ is the orthoclosure of the $\perp$-set $\{ \iota_{I_1}(x_1), \ldots, \iota_{I_n}(x_n) \}$. By Lemmas~\ref{lem:C-consists-of-Dacey-spaces} and~\ref{lem:rank-of-subspaces-of-Dacey-spaces}(i), $X$ has rank $n$.
\end{proof}

\begin{lemma} \label{lem:morphisms-between-any-pair-of-orthosets}
Let $X, Y \in \C$ and assume that $X$ has finite rank.
\begin{itemize}

\item[\rm (i)] The ranks of $X$ and $Y$ coincide if and only if there is a dagger isomorphism from $X$ to $Y$.

\item[\rm (ii)] The rank of $X$ is at most the rank of\/ $Y$ if and only if there is an isometry from $X$ to $Y$. In this case, any subspace of $Y$ that has the same rank as $X$ is the image of some isometry from $X$ to $Y$.

\end{itemize}
\end{lemma}

\begin{proof}
Let $m$ be the rank of $X$ and $n$ the rank of $Y$.

Ad (i): If $m=n$, both $X$ and $Y$ are, by Lemma~\ref{lem:nI}(ii), dagger biproducts of $n$ singletons. By Lemma~\ref{lem:nI}(i) and (H3)(b), two singletons are dagger isomorphic. By Lemma~\ref{lem:piecewise-dagger-isomorphy} there is hence a dagger isomorphism between $X$ and $Y$.

Conversely, if $X$ and $Y$ are dagger isomorphic, $X$ and $Y$ are orthoisomorphic by Lemma~\ref{lem:dagger-isos-in-C}(i) and have hence coinciding ranks.

Ad (ii): Assume $m \leq n$. Let $B$ be a subspace of $Y$ that has rank $m$. Note that, by Lemma~\ref{lem:rank-of-subspaces-of-Dacey-spaces}(iii), such a subspace always exists. By part (i), there is a dagger isomorphism $h \colon X \to B$ and by (H2), $Y = B \oplus B\c$. Then $\iota_B \circ h$ is an isometry from $X$ to $Y$ whose image is $B$.

Conversely, let $f \colon X \to Y$ be an isometry. By Lemma~\ref{lem:dagger-isos-in-C}(iii), $f$ preserves the orthogonality relation and hence $m \leq n$.
\end{proof}

\begin{proposition} \label{prop:nI}
For any $n \in \Naturals$, $\C$ contains up to dagger isomorphism exactly one orthoset of rank $n$.
\end{proposition}

\begin{proof}
As the orthoset of rank $0$ is the zero object, the assertion is clear in case $n = 0$.

Moreover, as we have assumed $\C$ to contain a non-zero orthoset, $\C$ contains by Lemma~\ref{lem:unital}(i) a unital orthoset. Hence, by Lemma~\ref{lem:nI}, $\C$ contains, for any $n \geq 1$, an orthoset of rank $1$. By Lemma~\ref{lem:morphisms-between-any-pair-of-orthosets}, any two orthosets of rank $n$ are dagger isomorphic.
\end{proof}

Let us fix a singleton orthoset $I \in \C$. For $n \in \Naturals$, let $n I$ be the $n$-fold biproduct of $I$, the case $n=0$ being understood as the zero orthoset. Then the finite-rank objects in $\C$ are, up to dagger isomorphism, exactly the orthosets $n I$, $n \in \Naturals$. Note also that $mI \oplus nI = (m+n)I$ for any $m,n \in \Naturals$.

We next turn to the partial orthometries in $\C$.

\begin{lemma} \label{lem:unique-generalised-inverse-in-C}
Assume that the $\C$-morphism $f \colon X \to Y$ is a partial orthometry and $f\adj$ is its generalised inverse. Then $f\adj$ is in $\C$ the unique generalised inverse of $f$.
\end{lemma}

\begin{proof}
Let $A = \image f\adj$ and $B = \image f$. Then $f\adj|_B^A = (f|_A^B)^{-1}$ and $\kernel f\adj = B\c$, that is, $f\adj \circ \iota_B = \iota_A \circ (f|_A^B)^{-1}$ and $f\adj \circ \iota_{B\c} = 0_{B\c,X}$. For any further generalised inverse $g$ of $f$, we likewise have $g \circ \iota_B = \iota_A \circ (f|_A^B)^{-1}$ and $g \circ \iota_{B\c} = 0_{B\c,X}$. As $Y = B \oplus B\c$, the assertion follows.
\end{proof}

\begin{lemma} \label{lem:partial-isometries-in-C}
For a $\C$-morphism $f \colon X \to Y$, the following are equivalent:
\begin{itemize}

\item[\rm (a)] $f$ is a partial isometry.

\item[\rm (b)] $f$ is a partial orthometry and $f\adj$ is the generalised inverse of $f$.

\item[\rm (c)] $f = h \oplus 0_{C,D}$, where $h$ is a dagger isomorphism and $0_{C,D}$ is any zero map.

\item[\rm (d)] There is in $\C$ a dagger isomorphism $h$ between a subspace $A$ of $X$ and a subspace $B$ of\/ $Y$ such that $f = \iota_B \circ h \circ {\iota_A}\adj$, where $\iota_A \colon A \to X$ and $\iota_B \colon B \to Y$ are the inclusion maps.

\end{itemize}
\end{lemma}

\begin{proof}
(a) $\Rightarrow$ (b): We will show that $\image f$ is orthoclosed. In a similarly way, we can see that also $\image f\adj$ is orthoclosed and (b) will then follow by Proposition~\ref{prop:partial-isometries}(i).

By Lemma~\ref{lem:image-of-partial-isometry}, $\image P(f)$ is an orthoclosed subset of $P(Y)$. Moreover, $f = f \circ f\adj \circ f$ implies that $\image f = f(\image f\adj)$. Note also that $f\adj = f\adj \circ f \circ f\adj$ and hence $(f\adj \circ f)\big|_{\image f\adj} = \id_{\image f\adj}$.

Let now $y \in Y$ and $x = f\adj(y)$. We claim that $f$ maps $\class{x}$ onto $\class{f(x)}$. This is clear if $x = 0$. Assume $x \neq 0$. Then also $f(x) \neq 0$. Let $\iota_{\{x\}\cc} \colon \{x\}\cc \to X$ and $\iota_{\{f(x)\}\cc} \colon \{f(x)\}\cc \to Y$ be the inclusion maps. Then $f\big|_{\{x\}\cc}^{\{f(x)\}\cc} = {\iota_{\{f(x)\}\cc}}\adj \circ f \circ \iota_{\{x\}\cc}$ is a dagger monomorphism and hence, by (H3)(a), bijective. The claim follows.

(b) $\Rightarrow$ (c): Let $f$ be a partial orthometry whose generalised inverse is $f\adj$. Let again $A = (\kernel f)\c$ and $B = \image f$, so that $f$ and $f\adj$ establish mutually inverse orthoisomorphisms between $A$ and $B$. Let $h = f|_A^B$. By Lemma~\ref{lem:restriction-corestriction-in-C}(i), $h\adj = f\adj|_B^A$ and hence $h$ is a dagger isomorphism. The commutative diagram
\[ \begin{tikzcd}
B \arrow[r, "\iota_B"]
& Y & B\c \arrow[l, "\iota_{B\c}"'] \\
A \arrow[u, "h"'] \arrow[r, "\iota_A"']
& X \arrow[u, "f"'] & A\c \arrow[l, "\iota_{A\c}"] \arrow[u, "0_{A\c,B\c}"']
\end{tikzcd} \]
shows $f = h \oplus 0_{A\c,B\c}$.

(c) $\Rightarrow$ (d): Let $h \colon A \to B$ and $f = h \oplus 0_{C,D}$ be as indicated. Then $X = A \oplus C$ and $Y = B \oplus D$. From the commutative diagram
\[ \begin{tikzcd}
B \arrow[r, "\iota_B"]
& Y & D \arrow[l, "\iota_{D}"'] \\
A \arrow[u, "h"'] \arrow[r, "\iota_A"']
& X \arrow[u, "f =", "h \oplus 0_{C,D}"' {yshift=-1pt}] & C \arrow[l, "\iota_{C}"] \arrow[u, "0_{C,D}"' {yshift=-2pt}]
\end{tikzcd} \]
we observe $(\iota_B \circ h \circ {\iota_A}\adj) \circ \iota_A = \iota_B \circ h = f \circ \iota_A$ and $(\iota_B \circ h \circ {\iota_A}\adj) \circ \iota_{C} = 0_{C,Y} = f \circ \iota_{C}$. Hence $f = \iota_B \circ h \circ {\iota_A}\adj$.

(d) $\Rightarrow$ (a): Let $h \colon A \to B$ be as indicated and $f = \iota_B \circ h \circ {\iota_A}\adj$. Then $\image f = \iota_B(h({\iota_A}\adj(X))) = \iota_B(h(A)) = B$ and $\image f\adj = \iota_A(h\adj({\iota_B}\adj(Y))) = \iota_A(h\adj(B)) = \iota_A(h^{-1}(B)) = A$. We easily verify $f \circ f\adj \circ f = f$ and thus (a) follows by Proposition~\ref{prop:partial-isometries}(i).
\end{proof}

\begin{lemma} \label{lem:isometries-in-C}
For a $\C$-morphism $f \colon X \to Y$, the following are equivalent:
\begin{itemize}

\item[\rm (a)] $f$ is an isometry.

\item[\rm (b)] $f$ is an orthometry and $f\adj$ is the generalised inverse of $f$.

\item[\rm (c)] $f = h \oplus 0_{0,B}$, where $h$ is a dagger isomorphism and $0_{0,B}$ is a zero map with domain $0$.

\item[\rm (d)] There is in $\C$ a dagger isomorphism $h$ between $X$ and a subspace $B$ of\/ $Y$ such that $f = \iota_B \circ h$, where $\iota_B \colon B \to Y$ is the inclusion map.

\end{itemize}
\end{lemma}

\begin{proof}
Each statement (a)--(d) is equivalent to the corresponding statement in Lem\-ma~\ref{lem:partial-isometries-in-C} together with the assumption that $f$ is injective.
\end{proof}

\begin{lemma} \label{lem:projection-C}
A $\C$-endomorphism $f$ is a projection if and only if $f$ is idempotent and self-adjoint.
\end{lemma}

\begin{proof}
If $f$ is idempotent and self-adjoint, $f$ is a partial isometry and hence $\image f$ is orthoclosed by Lemma~\ref{lem:partial-isometries-in-C}. Therefore the assertion holds by Lemma~\ref{lem:projections}.
\end{proof}

\begin{lemma} \label{lem:transitivity-of-orthoset}
Let $x, y$ be orthogonal proper elements of an orthoset $X \in \C$. Then there is in $\C$ a dagger automorphism $f$ of $X$ such that $f(\{x\}\cc) = \{y\}\cc$, \linebreak $f(\{y\}\cc) = \{x\}\cc$, and $f(z) = z$ for any $z \in \{x,y\}\c$.
\end{lemma}

\begin{proof}
As $X$ is a Dacey space, $(\{ x \}\cc, \{ y \}\cc, \{ x, y \}\c)$ is a decomposition of $X$. Hence, by Lemma~\ref{lem:extended-dagger-biproduct}(i), $X = \{ x \}\cc \oplus \{ y \}\cc \oplus \{ x, y \}\c$. Moreover, by (H3)(b), $\{x\}\cc$ and $\{y\}\cc$ are dagger isomorphic. Hence, by Lemma~\ref{lem:piecewise-dagger-isomorphy}, there is a dagger automorphism of $X$ as required.
\end{proof}

We conclude the section by describing the objects of $\C$ as uniform orthomodular spaces.

\begin{lemma} \label{lem:CX-in-C}
For any $X \in \C$, ${\mathsf C}(X)$ is a complete, irreducible, AC orthomodular lattice.
\end{lemma}

\begin{proof}
By Lemma~\ref{lem:C-consists-of-Dacey-spaces}, $X$ is an atomistic Dacey space and ${\mathsf C}(X)$ is an AC orthomodular lattice. Clearly, ${\mathsf C}(X)$ is complete.

Assume, for sake of contradiction, that ${\mathsf C}(X)$ is directly decomposable. By Lem\-ma~\ref{lem:unital}(ii), $\{x\}\cc$, $x \in X\withoutzero$, are the atoms of ${\mathsf C}(X)$. As the map $X \to {\mathsf C}(X) \komma x \mapsto \{x\}\cc$ reflects $\perp$, it follows that $X$ is reducible.

Hence there is a decomposition $(A,B)$ of $X$, where $A$ and $B$ are non-zero subspaces such that $A \cup B = X$. Let $x \in A$ and $y \in B$. As $X$ is an atomistic Dacey space, we have $\{x,y\}\cc \cap A = (\{x\}\cc \vee \{y\}\cc) \cap A = \{x\}\cc = \class x \cup \{0\}$ and similarly $\{x,y\}\cc \cap B = \{y\}\cc = \class y \cup \{0\}$. Hence $\{x,y\}\cc = \class x \cup \class y \cup \{0\}$.

Given a dagger isomorphism $f \colon \{x\}\cc \to \{y\}\cc$, there is by (H2) a map $d$ such that the diagram
\[ \begin{tikzcd}
& \{x\}\cc \arrow[dl, "\id_{\{x\}\cc}"', shift right=0.5ex] \arrow[d, "d"] \arrow[dr, "f", shift left=0.5ex] \\
\{x\}\cc 
& \{x,y\}\cc \arrow[l, "{\iota_{\{x\}\cc}}\adj", shift right=0.5ex] \arrow[r, "{\iota_{\{y\}\cc}}\adj"', shift left=0.5ex]
& \{y\}\cc
\end{tikzcd} \]
commutes. This is, however, impossible. We conclude that ${\mathsf C}(X)$ is irreducible.
\end{proof}

\begin{theorem} \label{thm:C-induces-star-sfield}
For each $X \in \C$ of rank $n$, there is an $n$-dimensional uniform orthomodular space $H$ and an orthoisomorphism between $P(X)$ and $P(H)$.

Moreover, let $h \colon X_1 \to X_2$ be an isometry. Then there are uniform orthomodular spaces $H_1$ and $H_2$ such that $H_1$ is a subspace of $H_2$, and orthoisomorphisms $r_1 \colon P(X_1) \to P(H_1)$ and $r_2 \colon P(X_2) \to P(H_2)$ such that the diagram
\begin{equation} \label{fml:C-induces-star-sfield}
\begin{tikzcd}
P(X_1) \arrow[r, "\cong", "r_1"'] \arrow[d, "P(h)"]
& P(H_1)  \arrow[d, "P(\iota_{H_1})"] \\
P(X_2) \arrow[r, "\cong", "r_2"'] & P(H_2) \text{ .}
\end{tikzcd}
\end{equation}
commutes, where $\iota_{H_1} \colon H_1 \to H_2$ is the inclusion map.
\end{theorem}

\begin{proof}
Assume first that $X \in \C$ has rank $n \geq 4$. By Lemma~\ref{lem:CX-in-C}, ${\mathsf C}(X)$ is a complete, irreducible, AC orthomodular lattice, which has length $\geq 4$. Hence, by Theorem~\ref{thm:orthomodular-space-as-orthoset}, there is an $n$-dimensional orthomodular space $H$ such that ${\mathsf C}(X)$ is isomorphic to ${\mathsf C}(H)$. By Lemma~\ref{lem:rescaling}, $H$ can be chosen to contain a unit vector. It follows that $P(X)$ is orthoisomorphic with $P(H)$.

We claim that $H$ is transitive. By Lemma~\ref{lem:transitivity-and-unit-vectors}, this implies that $H$ is uniform. By Lemma~\ref{lem:transitivity-of-orthoset}, there is for any 
$x, y \in X\withoutzero$ a dagger automorphism interchanging $\{x\}\cc$ and $\{y\}\cc$ and keeping $\{ x, y \}\c$ fixed. That is, there is an orthoautomorphism of $P(X)$ interchanging $\class x$ and $\class y$ and keeping $\{ \class x, \class y \}\c$ fixed. Consequently, for any orthogonal non-zero vectors $u, v \in H$, there is an orthoautomorphism of $P(H)$ interchanging $\class u$ and $\class v$ and keeping $\{\class u, \class v\}\c$ fixed. By Corollary~\ref{cor:Wigner-auto}, there is a unitary map on $H$ sending $\lin u$ to $\lin v$. The claim follows.

Let now $h \colon X_1 \to X_2$ be an isometry and assume that there is a uniform orthomodular space $H_2$ and an orthoisomorphism $r_2 \colon P(X_2) \to P(H_2)$. By Lemma~\ref{lem:isometries-in-C}, $h$ is an orthometry and by Lemma~\ref{prop:partial-isometries}(ii) so is $P(h) \colon P(X_1) \to P(X_2)$. Consequently, $H_2$ possesses a subspace $H_1$ such that $P(H_1)$ is the image of $r_2 \circ P(h)$. Furthermore, there is an orthoisomorphism $r_1 \colon P(X_1) \to P(H_1)$ making (\ref{fml:C-induces-star-sfield}) commutative.

Assume finally that $X \in \C$ is of rank $\leq 3$. By Lemma~\ref{lem:morphisms-between-any-pair-of-orthosets}(ii), there is an orthoset $Y \in \C$ of rank $4$ and an isometry $h \colon X \to Y$. It is now clear that there is a uniform orthomodular space $H$ and an orthoisomorphism $r \colon P(X) \to P(H)$.
\end{proof}


\section{The functor $\L \colon \C \to \OMS{F}$}
\label{sec:the-functor-L}

We will establish in this section that there is a $\ast$-sfield $F$ and a dagger functor $\L$ from $\C$ to the dagger category $\OMS{F}$. In particular, we will see that the orthosets in $\C$ arise from uniform orthomodular spaces over one and the same $\ast$-sfield. We will start by providing the details of the construction of the functor $\L$ as the precise procedure will matter in subsequent proofs. Afterwards we will investigate the properties of $\L$.

We begin by fixing an orthoset $Z \in \C$ of rank $3$. By Theorem~\ref{thm:C-induces-star-sfield}, there is a $3$-dimensional uniform orthomodular space $H^Z$ and an orthoisomorphism $r \colon P(Z) \to P(H^Z)$. We will denote its scalar $\ast$-sfield throughout this section by~$F$.

Let now $X$ be an arbitrary orthoset in $\C$. Consider the dagger biproduct
\[ \begin{tikzcd} X \arrow[r, "\iota_X"] & X \oplus Z & Z \arrow[l, "\iota_Z"'] \end{tikzcd} \]
and recall that, thanks to Lemma~\ref{lem:extended-dagger-biproduct}(ii), we may consider $X$ and $Z$ as complementary subspaces of $X \oplus Z$. By Lemma~\ref{lem:PA-subspace-of-PX}, $P(X)$ and $P(Z)$ are complementary subspaces of $P(X \oplus Z)$ and by Lemma~\ref{lem:extended-dagger-biproduct}(ii) and Proposition~\ref{prop:partial-isometries}(ii), the inclusion maps $P(\iota_X) \colon P(X) \to P(X \oplus Z)$ and $P(\iota_Z) \colon P(Z) \to P(X \oplus Z)$ are orthometries. Again by Theorem~\ref{thm:C-induces-star-sfield}, there is a uniform orthomodular space $\overline{H_X}$ and an orthoisomorphism $\bar r_X \colon P(X \oplus Z) \to P(\overline{H_X})$. Thanks to Lemma~\ref{lem:extension-of-representation}, we can assume that $H^Z$ is a subspace of $\overline{H_X}$ and $\bar r_X \circ P(\iota_Z) \circ r^{-1}$ is induced by the inclusion map $\iota_{H^Z} \colon H^Z \to \overline{H_X}$. Finally, $\bar r_X$ establishes an orthoisomorphism $r_X \colon P(X) \to P(H_X)$, where $H_X$ is the orthocomplement of $H^Z$ in $\overline{H_X}$. Note that then $\overline{H_X} = H_X \oplus H^Z$. We put $\L(X) = H_X$. In what follows, we will usually keep on writing $H_X$ rather than $\L(X)$. By Lemma~\ref{lem:direct-sum-of-orthomodular-spaces}, $H_X \in \OMS{F}$.

All this is summarised in the following commutative diagram:
\[ \begin{tikzcd}
P(Z) \arrow[r, "\cong", "r"'] \arrow[d, hook, "P(\iota_Z)"']
& P(H^Z) \arrow[d, hook, "P(\iota_{H^Z})"] \\
P(X \oplus Z) \arrow[r, "\cong", "\bar r_X"'] & P(H_X \oplus H^Z) \\
P(X) \arrow[r, "\cong", "r_X"'] \arrow[u, hook, "P(\iota_X)"]
& P(H_X) \arrow[u, hook, "P(\iota_{H_X})"']
\end{tikzcd} \]
Let now $f \colon X \to Y$ be a $\C$-morphism. Consider the morphism
\[ f \oplus \id_Z \colon X \oplus Z \to Y \oplus Z. \]
By Corollary~\ref{cor:Fundamental-Theorem}, there is a unique linear map $\bar\phi \colon H_X \oplus H^Z \to H_Y \oplus H^Z$ such that $P(\bar\phi) = \bar r_Y \circ P(f \oplus \id_Z) \circ {{\bar r}_X}^{\;\;-1}$ and $\bar\phi|_{H^Z} = \id_{H^Z}$. Note that $\bar\phi(H_X) \subseteq H_Y$ and hence we may define $\phi = \bar\phi|_{H_X}^{H_Y}$. In other words, $\phi \colon H_X \to H_Y$ is the unique linear map such that $\bar r_Y \circ P(f \oplus \id_Z) \circ {{\bar r}_X}^{\;\;-1} = P(\phi \oplus \id_{H^Z})$. We put $\L(f) = \phi$.

These facts are reflected in the following commutative diagrams:
\begin{equation} \label{fml:commutative-diagrams-for-Lf}
\begin{tikzcd}[row sep=large]
P(X \oplus Z) \arrow[r, "\cong", "\bar r_X"'] \arrow[d, "P(f \oplus \id_Z)"']
& P(H_X \oplus H^Z)  \arrow[d, "P(\bar\phi) \;=\; P(\phi \oplus \id_{H^Z})"'] \\
P(Y \oplus Z) \arrow[r, "\cong", "\bar r_Y"'] & P(H_Y \oplus H^Z)
\end{tikzcd}
\quad
\begin{tikzcd}[row sep=large]
H_X \arrow[d, "\phi"'] \arrow[r, hook, "\iota_{H_X}"] & H_X \arrow[d, "\bar\phi \;=\; \phi \oplus \id_{H^Z}"'] \oplus H^Z & H^Z \arrow[d, "\id_{H^Z}"'] \arrow[l, hook', "\iota_{H^Z}"'] \\
H_Y \arrow[r, hook, "\iota_{H_Y}"] & H_Y \oplus H^Z & H^Z \arrow[l, hook', "\iota_{H^Z}"']
\end{tikzcd}
\end{equation}
Disregarding $Z$ and $H^Z$, we have the commutative diagram
\begin{equation} \label{fml:restricted-commutative-diagram-for-Lf}
\begin{tikzcd}
P(X) \arrow[r, "\cong", "r_X"'] \arrow[d, "P(f)"]
& P(H_X)  \arrow[d, "P(\phi)"] \\
P(Y) \arrow[r, "\cong", "r_Y"'] & P(H_Y) \text{ .}
\end{tikzcd}
\end{equation}

\begin{proposition} \label{prop:L-is-functor}
$\L$ is a dagger functor from $\C$ to $\OMS{F}$.
\end{proposition}

\begin{proof}
The functoriality is readily checked. To see that $\L$ preserves the dagger, let $f \colon X \to Y$ be a morphism. Let $\phi = \L(f)$, that is, let $\phi \colon H_X \to H_Y$ be the unique linear map such that $\bar r_Y \circ P(f \oplus \id_Z) \circ \bar r_X^{-1} = P(\phi \oplus \id_{H^Z})$. By Proposition~\ref{prop:unitary-maps}, $\bar r_X$ is an adjoint of $\bar r_X^{-1}$ and $\bar r_Y^{-1}$ is an adjoint of $\bar r_Y$. As $P(f \oplus \id_Z)$ is adjointable, we conclude by Proposition~\ref{prop:quotient-map} that $P(\phi \oplus \id_{H^Z})$ possesses the unique adjoint $P(\phi \oplus \id_{H^Z})\adj = \bar r_X \circ P(f \oplus \id_Z)\adj \circ \bar r_Y^{-1}$.

By \cite[Theorem~3.23]{PaVe5}, the adjointability of $P(\phi \oplus \id_{H^Z})$ implies that $\phi \oplus \id_{H^Z}$ possesses a linear adjoint. From Lemma~\ref{lem:linear-adjoints-of-sums} it follows that also $\phi$ possesses a linear adjoint $\phi\adj$.

By Proposition~\ref{prop:quotient-map} and Lemma~\ref{lem:f-oplus-isomorphy}, we have $P(f \oplus \id_Z)\adj = P((f \oplus \id_Z)\adj) = P(f\adj \oplus \id_Z)$, and by Theorem~\ref{thm:adjoint-of-map-induced-by-linear-map} and Lemma~\ref{lem:linear-adjoints-of-sums}, $P(\phi \oplus \id_{H^Z})\adj = P((\phi \oplus \id_{H^Z})\adj) = P(\phi\adj \oplus \id_{H^Z})$. We conclude $\L(f\adj) = \phi\adj = \L(f)\adj$.
\end{proof}

\begin{proposition} \label{prop:L}
\begin{itemize}

\item[\rm (i)] For any orthoset $X \in \C$ of rank $n$, $H_X$ is $n$-dimensional.

\item[\rm (ii)] If $f \colon X \to Y$ is a dagger isomorphism, then $\L(f)$ is unitary.

\item[\rm (iii)] If $f \colon X \to Y$ is an isometry, then $\L(f)$ is a linear isometry.

\item[\rm (iv)] If $p \colon X \to X$ is a projection, then $\L(p)$ is a projection.

\item[\rm (v)] Let $\begin{tikzcd}[cramped] A \arrow[r, "\iota_A"] & X & B \arrow[l, "\iota_B"'] \end{tikzcd}$ be a biproduct and let
\begin{equation} \label{fml:prop:L}
\begin{tikzcd}[cramped] H_A \arrow[r, "\phi_A"] & H_X & H_B \arrow[l, "\phi_B"'] \end{tikzcd}
\end{equation}
be its image under $\L$. Then {\rm (\ref{fml:prop:L})} is a biproduct in $\OMS{F}$. That is, $\phi_A$ and $\phi_B$ are isometries whose images are complementary subspaces of $H_X$.

\end{itemize}
\end{proposition}

\begin{proof}
Ad (i): $P(X)$ and $P(H_X)$ are orthoisomorphic, hence the rank of $X$ coincides with the dimension of $H_X$.

Ad (ii), (iii): This follows from the functoriality of $\L$ and Lemma~\ref{lem:unitary-maps-and-isometries}.

Ad (iv): This follows from Lemma~\ref{lem:projection-C}, the functoriality of $\L$, and Lemma~\ref{lem:projections-in-orthomodular-spaces}.

Ad (v): Since $\iota_A$ and $\iota_B$ are isometries, $\phi_A$ and $\phi_B$ are by part (iii) linear isometries. According to (\ref{fml:restricted-commutative-diagram-for-Lf}), we have the commutative diagram
\[ \begin{tikzcd}
P(A) \arrow[d, "\cong" {rotate=270, anchor=south}, "r_A"'] \arrow[r, "P(\iota_A)"]
& P(X) \arrow[d, "\cong" {rotate=270, anchor=south}, "r_X"']
& P(B) \arrow[d, "\cong" {rotate=270, anchor=south}, "r_B"'] \arrow[l, "P(\iota_B)"'] \\
P(H_A)  \arrow[r, "P(\phi_A)"']
& P(H_X) & P(H_B) \text{ .} \arrow[l, "P(\phi_B)"]
\end{tikzcd} \]
By Lemmas~\ref{lem:extended-dagger-biproduct}(ii) and~\ref{prop:partial-isometries}(ii), $\image P(\iota_A)$ and $\image P(\iota_B)$ are complementary subspaces of $P(X)$. Hence $\image P(\phi_A)$ and $\image P(\phi_B)$ are complementary subspaces of $P(H_X)$. This means that $\image \phi_A$ and $\image \phi_B$ are complementary subspaces of $H_X$. As $H_X$ is orthomodular, this means $H_X = \image \phi_A \,\oplus\, \image \phi_B$. As dagger biproducts are given in $\OMS{F}$ by direct sums and as $\phi_A|^{\image \phi_A}$ and $\phi_B|^{\image \phi_B}$ are dagger isomorphisms in $\OMS{F}$, the assertion follows.
\end{proof}

Let us denote by $\Cfin$ be full subcategory of $\C$ consisting of orthosets of finite rank. Similarly, we denote by $\OMSfin{F}$ full subcategory of $\OMS{F}$ consisting of finite-dimensional spaces. Obviously, $\L$ restricts to a dagger functor $\Lfin \colon \Cfin \to \OMSfin{F}$. If $F$ is not one of the classical fields, $\OMS{F}$ does not contain spaces of infinite dimensions by Theorem~\ref{thm:Soler}(iii). In this case, we hence have $\Cfin = \C$ and $\OMSfin{F} = \OMS{F}$, so in particular $\Lfin$ coincides with $\L$.

\begin{proposition} \label{prop:L-is-surjective}
\begin{itemize}

\item[\rm (i)] $\Lfin$ is dagger essentially surjective.

\item[\rm (ii)] Assume that $\C$ contains an orthoset of infinite rank. Then $\L$ is dagger essentially surjective.

\end{itemize}
\end{proposition}

\begin{proof}
Ad (i): By Proposition~\ref{prop:nI}, $\C$ contains for any $n \in \Naturals$ an orthoset $X$ of rank $n$ and by Theorem \ref{prop:L-is-functor}(i), $\L(X)$ is $n$-dimensional. Moreover, by Example \ref{ex:categories}-1, any $n$-dimensional uniform orthomodular space $H$ is isomorphic to $F^n$, equipped with the standard inner product. Therefore any two $n$-dimensional uniform orthomodular spaces are isomorphic and hence dagger isomorphic in $\OMS{F}$.

Ad (ii): In this case, $\C$ contains an orthoset $X$ of any rank $n \in \Naturals \cup \{\aleph_0\}$. By Theorem~\ref{thm:Soler}(iii), $\L(X)$ is an $n$-dimensional Hilbert space over $F \in \{\Reals, \Complexes, \Quaternions\}$. Any two Hilbert spaces over $F$ are isomorphic if their dimensions coincide, hence they are dagger isomorphic in $\OMS{F}$.
\end{proof}

We next turn to the question which morphisms are in the image of $\L$. We insert an auxiliary lemma.

\begin{lemma} \label{lem:preservation-of-identity-parts}
Let $A$ be a subspace of $X \in \C$ and let $f \colon X \to X$ be a $\C$-endomorphism such that $f|_A = \id_A$ and $f(A\c) \subseteq A\c$. Let $S$ be the subspace of $H_X$ such that $r_X(P(A)) = P(S)$. Then $\L(f)$ is on $S$ the identity.
\end{lemma}

\begin{proof}
By Lemma \ref{lem:decomposing-a-morphism}(i), $f = \id_A \oplus f|_{A\c}^{A\c}$. Hence $f \oplus \id_Z = \id_A \oplus f|_{A\c}^{A\c} \oplus \id_Z$ is the identity on the subspace $A \oplus Z$ of $X$. Let $\phi = \L(f)$. In view of (\ref{fml:commutative-diagrams-for-Lf}) (with $Y = X$), we conclude that $P(\phi \oplus \id_{H^Z})$ is the identity on $\bar r_X(P(A \oplus Z)) = P(S \oplus H^Z)$. Hence $\phi \oplus \id_{H^Z}$ is the identity on $S \oplus H^Z$, in particular $\phi$ is the identity on $S$.
\end{proof}

Let $H_1, H_2$ be orthomodular spaces in the image of $\L$. We denote by $\E(H_1,H_2)$ the set of linear maps from $H_1$ to $H_2$ in the image of $\L$, and we write $\E(H_1) = \E(H_1,H_1)$.

\begin{proposition} \label{prop:image-L-contains-certain-maps}
Let $X, X_1, X_2 \in \C$ and put $H = \L(X)$, $\,H_1 = \L(X_1)$, and $H_2 = \L(X_2)$.
\begin{itemize}

\item[\rm (i)] Assume that $H_1$ and $H_2$ have the same finite dimension. Then $\E(H_1,H_2)$ contains a unitary map between $H_1$ to $H_2$.

\item[\rm (ii)] Assume that $H_1$ is finite-dimensional and let $S$ be a subspace of $H_2$ such that $\dim S = \dim H_1$. Then $\E(H_1,H_2)$ contains a linear isometry from $H_1$ to $H_2$ whose image is $S$.

\item[\rm (iii)] For each subspace $S$ of $H$, $\E(H)$ contains the projection from $H$ to $S$.

\item[\rm (iv)] $\E(H)$ is closed under composition.

\item[\rm (v)] $\E(H_1,H_2)$ is closed under addition.

\item[\rm (vi)] Let $u, v \in H \setminus \{0\}$ such that $u \perp v$. Then $\E(H)$ contains a unitary map on $H$ that interchanges $\lin u$ and $\lin v$ and keeps $\{u,v\}\c$ elementwise fixed.

\end{itemize}
\end{proposition}

\begin{proof}
Ad (i): This is clear from Lemma~\ref{lem:morphisms-between-any-pair-of-orthosets}(i) and Proposition~\ref{prop:L}(ii).

Ad (ii): This is clear from Lemma~\ref{lem:morphisms-between-any-pair-of-orthosets}(ii) and Proposition~\ref{prop:L}(iii).

Ad (iii): Let $A$ be the subspace of $X$ such that $r_X(P(A)) = P(S)$, and let $p = \iota_A \circ {\iota_A}\adj$. By Lemma~\ref{lem:inclusion-and-Sasaki-in-C}(ii), $p$ is the projection of $X$ onto $A$ and by Proposition~\ref{prop:L}(iv), $\pi = \L(p)$ is a projection as well. As $\image p = A$, we have, in view of (\ref{fml:restricted-commutative-diagram-for-Lf}), that $\image \pi = S$.

Ad (iv): This is clear from the functoriality of $\L$.

Ad (v): Since $\C$ has biproducts, $\C$ possesses by Theorem~\ref{thm:semiadditive-structure} a unique semiadditive structure. Let $f, g \colon X_1 \to X_2$ be $\C$-morphisms. We claim that then $\L(f+g) = \L(f) + \L(g)$.

Let $\phi = \L(f)$, and $\psi = \L(g)$. By Proposition~\ref{prop:L}(v), the image of (\ref{fml:addition}) under $\L$ is as follows:
\[ \begin{tikzcd}
& H_1 \arrow[dl, "\id"', shift right=0.5ex] \arrow[d, dashed, "\Delta"] \arrow[dr, "\id", shift left=0.5ex] \\
H_1 \arrow[r, "h_{1l}"', shift right=0.5ex] \arrow[d, "\phi"]
& H_{11} \arrow[l, "h_{1l}\adj"', shift right=0.5ex] \arrow[r, "h_{1r}\adj", shift left=0.5ex] \arrow[d, dashed, "\phi \oplus \psi"]
& H_1 \arrow[l, "h_{1r}", shift left=0.5ex] \arrow[d, "\psi"] \\
H_2 \arrow[r, "h_{2l}"'] \arrow[dr, "\id"', shift right=0.5ex]
& H_{22} \arrow[d, dashed, "\nabla"]
& H_2 \arrow[l, "h_{2r}"] \arrow[dl, "\id", shift left=0.5ex] \\
& H_2
\end{tikzcd} \]
Here, $h_{1l}$ and $h_{1r}$ are isometries whose images are complementary subspaces of the orthomodular space $H_{11}$, and similarly for $h_{2l}$ and $h_{2r}$. We conclude that $\L(f+g) = \nabla \circ (\phi \oplus \psi) \circ \Delta = \phi + \psi = \L(f) + \L(g)$.

Ad (vi): Let $x, y \in X$ be such that $r_X(\class x) = \class u$ and $r_X(\class y) = \class v$. By Lemma~\ref{lem:transitivity-of-orthoset}, there is a dagger automorphism $f \colon X \to X$ such that $\{x\}\cc$ and $\{y\}\cc$ are interchanged and $\{x,y\}\c$ is elementwise fixed. By Proposition~\ref{prop:L}(ii), $\phi = \L(f)$ is unitary. Moreover, $\phi$ interchanges $\lin u$ and $\lin v$, and by Lemma~\ref{lem:preservation-of-identity-parts}, $\phi$ is the identity on $\{u,v\}\c$.
\end{proof}

According to Proposition~\ref{prop:image-L-contains-certain-maps}, $\E(H)$ contains, for any space $H$ in the image of $\L$, all projections and is closed under addition as well as composition. Under a particular assumption, we can draw the conclusion that $\E(H)$ contains all endomorphisms.

Let $\Sym F = \{ \alpha \in F \colon \alpha\inv = \alpha \}$, which is the set of {\it symmetric} elements of $F$. Let $C(F) = \{ \alpha\inv \alpha \colon \alpha \in F \}$. Then $C(F)$ is a positive cone of the additive group $\Sym F$. Indeed, we clearly have $C(F) \subseteq \Sym F$. As by Lemma~\ref{lem:sfield-of-transitive-orthomodular-spaces} $F$ is Pythagorean and formally real, $\alpha, \beta \in C(F)$ implies $\alpha + \beta \in C(F)$, and $\alpha, -\alpha \in C(F)$ implies $\alpha = 0$. Putting $\alpha \leq \beta$ if $\beta - \alpha \in C(F)$, $\Sym F$ becomes consequently a partially ordered abelian group. Clearly, $C(F) = \{ \alpha \in \Sym F \colon \alpha \geq 0 \}$.

We note that $1 > 0$ and in fact $r > 0$ for all strictly positive rationals $r$. Moreover, $\alpha \geq 0$ implies $\beta\inv \alpha \beta \geq 0$ for any $\beta \in F$. In particular, $\alpha > 0$ implies $\alpha^{-1} = \alpha^{-1} \alpha \alpha^{-1} > 0$.

Let us call $\alpha \in C(F)$ {\it bounded} if there is an $n \in \Naturals$ such that $\alpha \leq n$. We will call $F$ {\it Archimedean} if any $\alpha \in C(F)$ is bounded. We note that in this case $0 \leq \alpha \leq \tfrac 1 n$ holds for all $n \in \Naturals \setminus \{0\}$ only if $\alpha = 0$. Obviously, the classical $\ast$-sfields $\Reals$, $\Complexes$, and $\Quaternions$ are Archimedean.

\begin{lemma} \label{lem:Archimedean-sfield}
Assume that $F$ is Archimedean. Let $K(F)$ be the set of all $\alpha \in F$ with the property that there is a $\beta \in F$ such that $\alpha \alpha\inv + \beta \beta\inv = 1$. Then each element of $F$ is a sum of elements of $K(F)$.
\end{lemma}

\begin{proof}
Let $\alpha \in F \setminus \{0\}$. Then $\alpha \alpha\inv > 0$ and hence there is an $n \in \Naturals \setminus \{0\}$ such that $\alpha \alpha\inv \leq n \leq n^2$. This means that there is a $\beta \in F$ such that $\alpha \alpha\inv + \beta \beta\inv = n^2$ and consequently $\tfrac{\alpha}{n}\big(\tfrac{\alpha}{n}\big)\inv + \tfrac{\beta}{n}\big(\tfrac{\beta}{n}\big)\inv = 1$. As $\alpha = n \, \tfrac \alpha n$, the assertion follows.
\end{proof}

\begin{remark}
Assume that $F$ is commutative. Then we may extend $C(F)$ to a set $\bar C(F) \subseteq \Sym F$ with the following properties: $\bar C(F) + \bar C(F) \subseteq \bar C(F)$, $\,\bar C(F) \alpha\inv \alpha \subseteq C(F)$ for any $\alpha \in F$, $1 \in \bar C(F)$, and $\bar C(F)$ contains, for any non-zero $\alpha \in \Sym F$, exactly one of $\alpha$ or $-\alpha$. Thus $F$ can be made into a Baer-ordered $\ast$-sfield. Consequently, if $F$ is Archimedean, then $F$ embeds into $\Reals$ or $\Complexes$ \cite[2.~Theorem]{Hol1}.
\end{remark}

\begin{proposition} \label{prop:L-is-full-for-certain-F}
Assume that $F$ is Archimedean. Let $H$ be a finite-dimensional orthomodular space in the image of \/ $\L$. Then $\E(H)$ contains all linear endomorphisms of $H$.
\end{proposition}

\begin{proof}
Let $n$ be the dimension of $H$. We will represent vectors and endomorphisms of $H$ w.r.t.\ some fixed ordered orthonormal basis $(e_1, \ldots, e_n)$. We denote the projection on a subspace $S$ by $P_S$. Note that, for a unit vector $u = \begin{sm} \alpha_1 \\ \vdots \\ \alpha_n \end{sm} \in H$, we have $P_{\lin u} = (\alpha_j\inv \alpha_i)_{i,j}$.

For any $A = (a_{ij})_{i,j} \in \E(H)$, we have $P_{\lin{e_i}} \, A \, P_{\lin{e_j}} = (\delta_{ii'} \delta_{jj'} a_{ij})_{i',j'}$. Moreover, for $u = \begin{sm} 1 \\ \vdots \\ 1 \end{sm}$, we get $n P_{\lin u} = (1)_{i,j}$. It follows that the unitary maps on $H$ that interchange two basis vectors are in $\E(H)$. We conclude that whenever each value $a_{ij}$, $i, j = 1, \ldots, n$, occurs as some entry in a matrix representing an endomorphism in $\E(H)$, also $(a_{ij})_{i,j} \in \E(H)$.

Let us now restrict to the case that $n \geq 3$. For $\alpha \in F$, denote by $N_\alpha$ the matrix that has $\alpha$ at position $(1,1)$ and else $0$'s. Considering the projection to $\lin{e_1+e_2}$, we see that $N_{\frac 1 2} \in \E(H)$.

Let $K(F)$ be defined as in Lemma~\ref{lem:Archimedean-sfield}. Let $\alpha \in K(F)$ and let $\beta \in F$ be such that $\alpha \alpha\inv + \beta \beta\inv = 1$. Then $u = \frac 2 3 \alpha \, e_1 + \frac 2 3 \beta \, e_2 + \frac 1 3 \, e_3$ is a unit vector. Among the entries of the matrix representing the projection to $\lin u$, we find 
$(\frac 1 3)\inv \, \frac 2 3 \alpha = \frac 2 9 \alpha$. Hence $N_{\frac 2 9 \alpha} \in \E(H)$ and hence also $9 N_{\frac 1 2} N_{\frac 2 9 \alpha} = N_{\alpha}$. This in turn implies that $\E(H)$ contains all endomorphisms whose matrices have entries in $K(F)$. By Lemma~\ref{lem:Archimedean-sfield}, $\E(H)$ consists of all endomorphisms.

Assume now that $H$ is $1$- or $2$-dimensional. Let $H'$ be a $3$-dimensional space in the image of $\L$. Then the image of $\L$ contains an isometry $\phi \colon H \to H'$. Let $M_\alpha$ be the multiplication by $\alpha \in F$ in $H'$. We claim that then $\phi\adj \, M_\alpha \, \phi$ is the multiplication by $\alpha$ in $H$. Indeed, for any $u, v \in H$, we have
\[ \herm{\phi\adj(\alpha\phi(u))}{v}
\;=\; \herm{\alpha\phi(u)}{\phi(v)}
\;=\; \alpha \herm{\phi(u)}{\phi(v)}
\;=\; \alpha \herm{u}{v}
\;=\; \herm{\alpha u}{v} \]
and hence $\phi\adj(\alpha\phi(u)) = \alpha u$. The above arguments now show that $\E(H)$ contains all endomorphisms.
\end{proof}

Our final theorem applies in particular to the case that $\C$ contains orthosets of infinite rank.

\begin{theorem} \label{thm:L-is-surjective}
Assume that $F$ is Archimedean. The $\Lfin$ is dagger essentially surjective and full.
\end{theorem}

\begin{proof}
This is clear from Propositions~\ref{prop:L-is-surjective},~\ref{prop:image-L-contains-certain-maps} and~\ref{prop:L-is-full-for-certain-F}.
\end{proof}

\section{Categories of orthosets: extended set of hypotheses}
\label{sec:extended-set-of-hypotheses}

So far we have established that $\C$, a dagger category of orthosets fulfilling the three properties (H1)--(H3), can be regarded as a dagger category of uniform orthomodular spaces. The scalar $\ast$-sfield $F$ is the same for all spaces and is one of $\Reals$, $\Complexes$, or $\Quaternions$ if $\C$ contains an orthoset of infinite rank. We have moreover defined a functor $\L \colon \C \to \OMS{F}$, sending orthosets to orthomodular spaces and adjointable maps to linear maps.

We shall now extend our set of hypotheses. The strengthening will turn out to imply faithfulness of $\L$. If $\C$ contains an orthoset of infinite rank, one further assumption will imply that $F = \Complexes$ and that $\L$ is also full. Consequently, we get a unitary dagger equivalence with $\Hil{\Complexes}$ in this case.

For convenience, we will restate all our conditions. $\C$ is assumed to fulfil the following conditions:
\begin{itemize}

\item[\rm (H1)] $\C$ has dagger biproducts.

\item[\rm (H2)] Let $X \in \C$ and let $A$ be a subspace of $X$. Then $A$ and $A\c$ as well as the inclusion maps $\iota_A \colon A \to X$ and $\iota_{A\c} \colon A\c \to X$ belong to $\C$ and
\[ \begin{tikzcd} A \arrow[r, "\iota_A"] & X & A\c \arrow[l, "\;\;\iota_{A\c}"'] \end{tikzcd} \]
is a dagger biproduct of $A$ and $A\c$.

\item[\rm (H3')] Any non-zero morphism between unital orthosets in $\C$ is an isomorphism.

\item[\rm (H4)] For any two objects $X, Y \in \C$, there is a dagger monomorphism from $X$ to $Y$ or from $Y$ to $X$.

\end{itemize}
Again, it is understood that $\C$ does not consist of the zero orthoset alone.

Note that conditions (H1) and (H2) are unmodified and that (H3') together with (H4) implies (H3). The $\star$-sfield $F$ and the functor $\L \colon \C \to \OMS{F}$ are still assumed to be given according to our specifications in Sections~\ref{sec:three-basic-hypotheses} and ~\ref{sec:the-functor-L}.

\begin{example} \label{ex:HilF-fulfils-H1-H4}
$\OMS{F}$ fulfils conditions {\rm (H1)}--{\rm (H4)}. Indeed, we have already seen in Example~\ref{ex:OMSF-fulfils-H1-H3} that $\OMS{F}$ fulfils {\rm (H1)} and {\rm (H2)}.

Ad {\rm (H3')}: A non-zero morphism between unital orthosets in $\OMS{F}$ is a non-zero linear map $\phi$ between $1$-dimensional spaces and hence a linear isomorphism. As $\phi$ and $\phi^{-1}$ are adjointable, $\phi$ is an isomorphism in $\OMS{F}$.

Ad {\rm (H4)}: Let $H_1$ and $H_2$ be uniform orthomodular spaces. If $H_1$ and $H_2$ are finite-dimensional, any one-to-one map from an orthonormal basis of one space to an orthonormal basis of the other one defines a linear isometry. If one of $H_1$ or $H_2$ is infinite-dimensional, then $\OMS{F}$ consists of Hilbert spaces and hence there is likewise a linear isometry between $H_1$ and $H_2$.
\end{example}

We show some lemmas depending on the new condition (H3').

\begin{lemma} \label{lem:isomorphism-of-singletons}
Let $f \colon X \to Y$ be a $\C$-morphism and let $x \in X \setminus \kernel f$. Then $f\big|_{\{x\}\cc}^{\{f(x)\}\cc}$ is an isomorphism. Consequently, $f$ establishes a bijection between $\class x$ and $\class{f(x)}$
\end{lemma}

\begin{proof}
By Lemma~\ref{lem:restriction-corestriction-in-C}, $f|_{\{x\}\cc}^{\{f(x)\}\cc}$ is a morphism, which is non-zero and hence by (H3') an isomorphism. We recall that $\{x\}\cc = \class x \cup \{0\}$, $\,\{f(x)\}\cc = \class{f(x)} \cup \{0\}$, and $f(0) = 0$. Thus $f$ maps $\class x$ bijectively to $\class{f(x)}$.
\end{proof}

We recall that we denote by $I$ a unital orthoset in $\C$.

\begin{lemma} \label{lem:dagger-mono-from-mono}
Let $X \in \C$ and let $f \colon I \to X$ be a non-zero morphism. Then $f$ is a monomorphism, whose image is a unital subspace $J$ of $X$. Moreover, there is an isomorphism $h \colon J \to I$ such that $f \circ h$ equals the inclusion map $\iota_J \colon J \to X$.
\end{lemma}

\begin{proof}
Let $x \in I\withoutzero$ and $J = \{f(x)\}\cc$. By Lemma~\ref{lem:isomorphism-of-singletons}, $f|^J$ is an isomorphism. In particular, $\image f = J$. As the inclusion map $\iota_J \colon J \to X$ is a dagger monomorphism, it follows that $f = \iota_J \circ f|^J$ is likewise a monomorphism. Moreover, putting $h = (f|^J)^{-1}$ we have $f \circ h = \iota_J \circ f|^J \circ h = \iota_J$.
\end{proof}

\begin{lemma} \label{lem:sums-parallel-to-identity}
Let $f, g \colon I \to I$ be such that $f \oplus g \prl \id_{I \oplus I}$. Then $f = g$.
\end{lemma}

\begin{proof}
Consider the biproduct $\begin{tikzcd}[cramped] I \arrow[r, "\iota_1"] & I \oplus I & I \arrow[l, "\iota_2"'] \end{tikzcd}$ and let $m \;=\; \iota_2 \circ {\iota_1}\adj \;+\; \iota_1 \circ {\iota_2}\adj$. By Lemmas~\ref{lem:morphism-plus-zero} and~\ref{lem:plus-oplus}, $m \circ m = \iota_1 \circ {\iota_1}\adj \;+\; \iota_2 \circ {\iota_2}\adj = (\id_I \oplus 0) + (0 \oplus \id_I) = \id_I \oplus \id_I = \id_{I \oplus I}$. By Lemma~\ref{lem:adjoints-and-plus}, $m$ is self-adjoint and hence $m$ is an orthoautomorphism of $I \oplus I$. A short calculation and a further application of Lemmas~\ref{lem:morphism-plus-zero} and~\ref{lem:plus-oplus} gives
\[ m \circ (f \oplus g) \circ m
\;=\; \iota_2 \circ f \circ {\iota_2}\adj + \iota_1 \circ g \circ {\iota_1}\adj
\;=\; (0 \oplus f) + (g \oplus 0)
\;=\; g \oplus f. \]
We conclude that $g \oplus f \prl \id_{I \oplus I}$.

Consider next the diagonal $\Delta \colon I \to I \oplus I$. Let $x \in I\withoutzero$ and put $J=\{ \Delta(x) \}\cc$. By Lemma~\ref{lem:dagger-mono-from-mono}, there is an isomorphism $h \colon J \to I$ such that $d_1 = \Delta \circ h$ is the inclusion map from $J$ to $I \oplus I$. Hence there is a dagger biproduct $\begin{tikzcd}[cramped] J \arrow[r, "d_1"] & I \oplus I & J^{\c} \arrow[l, "d_2"'] \end{tikzcd}$. We have $d_1\adj = h\adj \circ \nabla$ and thus
\[ d_1\adj \circ (f \oplus g) \circ \Delta
\;=\; h\adj \circ (f+g) \;=\; h\adj \circ (g+f) = d_1\adj \circ (g \oplus f) \circ \Delta. \]
Moreover, $d_2\adj \circ d_1 = 0$ implies $d_2\adj \circ \Delta = 0$. By Lemma~\ref{lem:dagger-mono-from-mono},
\[ \image \Delta \;=\; \{ \Delta(x) \}\cc \;=\; \class{\Delta(x)} \cup \{0\} \;\subseteq\; \kernel d_2\adj. \]
Recall that $f \oplus g \prl \id_{I \oplus I}$, hence $f \oplus g$ maps $\class{\Delta(x)}$ to itself. We conclude that also $d_2\adj \circ (f \oplus g) \circ \Delta = 0$. Similarly, we see that $d_2\adj \circ (g \oplus f) \circ \Delta = 0$. It follows that $(f \oplus g) \circ \Delta = (g \oplus f) \circ \Delta$. From the commutative diagram
\[ \begin{tikzcd}
& I \arrow[dl, "\id"', shift right=0.5ex] \arrow[d, "\Delta"] \arrow[dr, "\id"] \\
I \arrow[d, "f"]
& I \oplus I \arrow[l, "\pi_1"'] \arrow[r, "\pi_2"] \arrow[d, "f \oplus g"]
& I \arrow[d, "g"] \\
I
& I \oplus I \arrow[l, "\pi_1"'] \arrow[r, "\pi_2"]
& I
\end{tikzcd} \]
we observe that $\pi_1 \circ (f \oplus g) \circ \Delta = f \circ \pi_1 \circ \Delta = f$. Similarly, we have $\pi_1 \circ (g \oplus f) \circ \Delta = g \circ \pi_1 \circ \Delta = g$ and hence $f = g$.
\end{proof}

We are ready to show that $\L$ is faithful.

\begin{lemma} \label{lem:oplus-id-cancellation}
Let $f, g \colon X \to Y$ be in $\C$ and let $W \in \C$ be non-zero. If $f \oplus \id_W \prl g \oplus \id_W$, then $f = g$.
\end{lemma}

\begin{proof}
Assume $f \oplus \id_W \prl g \oplus \id_W$ and let $x \in X$. By assumption, $f = \pi_X \circ (f \oplus \id_W)\circ \iota_X \prl \pi_X \circ (g \oplus \id_W) \circ \iota_X = g$. We conclude that $\{f(x)\}\cc = \{g(x)\}\cc$. We have to show that in fact $f(x) = g(x)$. This is clear if $f(x) = 0$. Let us assume that $f(x)$ is proper. Note that then also $g(x)$ is proper.

Pick some $w \in W\withoutzero$, let $\tilde X = \{x\}\cc \oplus \{w\}\cc$ and $\tilde Y = \{f(x)\}\cc \oplus \{w\}\cc = \{g(x)\}\cc \oplus \{w\}\cc$. Furthermore, let $\tilde f = (f \oplus \id_W)\big|_{\tilde X}^{\tilde Y} = f\big|_{\{x\}\cc}^{\{f(x)\}\cc} \oplus \id_{\{w\}\cc}$ and similarly $\tilde g  = (g \oplus \id_W)\big|_{\tilde X}^{\tilde Y} = g\big|_{\{x\}\cc}^{\{f(x)\}\cc}  \oplus \id_{\{w\}\cc}$. Then $\tilde f \prl \tilde g$ and by (H3'), $\tilde f$ and $\tilde g$ are isomorphisms and we have
\[ {\tilde g}^{-1} \circ \tilde f \;=\; \Big(\big(g\big|_{\{x\}\cc}^{\{f(x)\}\cc}\big)^{-1} \circ f\big|_{\{x\}\cc}^{\{f(x)\}\cc}\Big) \oplus \id_{\{w\}\cc} \prl \id_{\{x\}\cc \oplus \{w\}\cc}. \]
By Lemma~\ref{lem:sums-parallel-to-identity}, $\big(g\big|_{\{x\}\cc}^{\{f(x)\}\cc}\big)^{-1} \circ f\big|_{\{x\}\cc}^{\{f(x)\}\cc} = \id_{\{x\}\cc}$. That is, $f(x) = g(x)$.
\end{proof}

\begin{lemma} \label{lem:L-is-faithful}
$\L$ is faithful.
\end{lemma}

\begin{proof}
Let $f, g \colon X \to Y$ be such that $\phi = \L(f) = \L(g)$. Then $P(f \oplus \id_Z) = P(g \oplus \id_Z)$, that is, $f \oplus \id_Z \prl g \oplus \id_Z$. Hence $f = g$ by Lemma~\ref{lem:oplus-id-cancellation}.
\end{proof}

\begin{theorem}
Assume that $F$ is Archimedean. Then $\Cfin$ and $\OMSfin{F}$ are unitarily dagger equivalent.
\end{theorem}

\begin{proof}
By Theorem~\ref{thm:L-is-surjective}, $\Lfin$ is dagger essentially surjective and full, and by Lem\-ma~\ref{lem:L-is-faithful}, $\Lfin$ is faithful. Hence, by \cite[Lemma~V.1]{Vic}, $\Lfin$ is a unitary dagger equivalence.
\end{proof}

We finally turn to the announced description of the category of complex Hilbert spaces and bounded linear maps.

Let $f \colon X \to X$ be a dagger automorphism in $\C$. We call $g \colon X \to X$ a {\it strict square root} of $f$ if $g$ is a dagger automorphism such that $g^2 = f$ and for any projection $p$ of $X$, $p$ commutes with $f$ if and only if $p$ commutes with $g$.

We add the following hypothesis:

\begin{itemize}

\item[\rm (H5)] Every dagger automorphism in $\C$ possesses a strict square root.

\end{itemize}

We moreover assume that $\C$ contains an orthoset of infinite rank. Recall that $F$ is in this case one of $\Reals$, $\Complexes$, or $\Quaternions$ and $\OMS{F}$ coincides with $\Hil{F}$, the dagger category of Hilbert spaces over $F$.

\begin{example} \label{ex:HilF-fulfils-H1-H5}
The dagger category $\Hil{\Complexes}$ of complex Hilbert spaces fulfils all conditions {\rm (H1)}--{\rm (H5)}. Indeed, $\Hil{\Complexes} = \OMS{\Complexes}$ and hence by Example~\ref{ex:HilF-fulfils-H1-H4}, $\Hil{\Complexes}$ fulfils {\rm (H1)}--{\rm (H4)}.

Ad {\rm (H5)}: The dagger automorphisms of $\Hil{\Complexes}$ are the unitary operators and the projections of $\Hil{\Complexes}$ are the projections in the usual sense. Hence {\rm (H5)} holds by Lemma~\ref{lem:strict-square-root-and-C}.
\end{example}

We next extend Proposition~\ref{prop:image-L-contains-certain-maps} by one more statement.

\begin{proposition} \label{prop:linear-maps-in-E}
Let $H$ be a Hilbert space in the image of $\L$. Then any unitary map $\phi \colon H \to H$ in $\E(H)$ has a strict square root in $\E(H)$.
\end{proposition}

\begin{proof}
Let $f \colon X \to X$ be such that $\phi = \L(f)$. As $\L$ is by Lemma~\ref{lem:L-is-faithful} faithful, $f$ is a dagger automorphism. By (H5), $f$ possesses a strict square root $g \in \C$. By Proposition~\ref{prop:image-L-contains-certain-maps}, $\E(H)$ contains all the projections of $H$. It follows that $\L(g)$ is a strict square root of $\L(f)$.
\end{proof}

We shall now show that the $\star$-sfield $F$ is actually the field of complex numbers, equipped with the complex conjugation. We put $U(F) = \{ \epsilon \in F \colon \epsilon\inv \epsilon = 1 \}$.

\begin{lemma} \label{lem:C-consists-of-complex-Hilbert-spaces}
$F = \Complexes$.
\end{lemma}

\begin{proof}
Let $X \in \C$ be of rank $2$ and let $H = \L(X)$. By Proposition~\ref{prop:L-is-full-for-certain-F}, the map $\phi \colon H \to H \komma u \mapsto -u$ is in $\E(H)$ and hence, by Proposition~\ref{prop:linear-maps-in-E}, possesses a square root. The assertion now follows from Lemma~\ref{lem:strict-square-root-and-C}.
\end{proof}

To prove the fullness of $\L$, we make use of a theorem of A.~M.~Bikchentaev \cite{Bik1}, see also \cite{Bik2,Bik3}.

\begin{lemma} \label{lem:all-endomorphisms}
Let $H$ be a Hilbert space in the image of \/ $\L$. Then for each endomorphism $\phi \colon H \to H$, there is a $\C$-morphism $f$ such that $\L(f) = \phi$.
\end{lemma}

\begin{proof}
By Proposition~\ref{prop:image-L-contains-certain-maps}, $\E(H)$ contains all projections and is closed under addition as well as composition. From \cite[Section 3, Theorem]{Bik1}, it follows that $\E(H)$ consists of all endomorphisms of $H$.
\end{proof}

\begin{lemma} \label{lem:L-is-full}
$\L$ is full.
\end{lemma}

\begin{proof}
Let $H_1, H_2$ be Hilbert spaces in the image of $\L$. By (H4), there is an isometry between $H_1$ and $H_2$ in the image of $\L$. Hence the assertion is a consequence of Lemma~\ref{lem:all-endomorphisms}.
\end{proof}

\begin{theorem}
The dagger categories $\C$ and $\Hil{\Complexes}$ are unitarily dagger equivalent.
\end{theorem}

\begin{proof}
$\L \colon \C \to \Hil{\Complexes}$ is by Proposition~\ref{prop:L-is-functor} a dagger functor. By Proposition~\ref{prop:L-is-surjective}, $\L$ is dagger essentially surjective. By Lemmas~\ref{lem:L-is-faithful} and~\ref{lem:L-is-full}, $\L$ is faithful and full. Hence the assertion follows by \cite[Lemma~V.1]{Vic}.
\end{proof}

\subsubsection*{Acknowledgements}

This research was funded in part by the Austrian Science Fund (FWF) 10.55776/ PIN5424624 and the Czech Science Foundation (GA\v CR) 25-20013L.

\end{document}